%% file: main.tex
\documentclass[11pt]{article}
\input{header}
\input{macro}

\title{New Results on the Polyak Stepsize: Tight Convergence Analysis and Universal Function Classes}

\author{
{Chang He} \thanks{The first two authors contribute equally. School of Information Management and Engineering, Shanghai University of Finance and Economics; Department of Industrial and Systems Engineering, University of Minnesota. \texttt{ischanghe@gmail.com}}
\and
{Wenzhi Gao} \thanks{ICME, Stanford University. \texttt{gwz@stanford.edu}}
\and
{Bo Jiang} \thanks{School of Information Management and Engineering, Shanghai University of Finance and Economics. \texttt{isyebojiang@gmail.com}}
\and
{Madeleine Udell} \thanks{ICME, Stanford University. \texttt{udell@stanford.edu}}
\and
{Shuzhong Zhang} \thanks{Department of Industrial and Systems Engineering, University of Minnesota. \texttt{zhangs@umn.edu}}
}

\begin{document}
\maketitle
\input{abstract.tex}
\input{sec_intro.tex}
\input{sec_lb.tex}
\input{sec_ub.tex}

\section{Conclusions}
This work revisits the classical Polyak stepsize and establishes several new results. We establish the tightness of the current convergence results and provide new convergence rates for new function classes. An interesting future direction is how to construct a universal worst-case function for general adaptive stepsizes, beyond just the Polyak stepsize. 

\bibliographystyle{abbrvnat}
\bibliography{ref}

\clearpage
\appendix

\input{sec_appendix.tex}

\end{document}

%% file: header.tex
\usepackage[top=1in, bottom=1in, left=1in, right=1in]{geometry}

\usepackage{enumitem}
\usepackage[linktocpage,colorlinks,linkcolor=blue,anchorcolor=blue,citecolor=blue,urlcolor=blue,pagebackref]{hyperref}
\usepackage{amsmath,amsthm}
\usepackage{algpseudocode}

\usepackage{accents}
\usepackage{comment,url,graphicx,relsize}
\usepackage{amssymb,amsfonts,amsmath,amsthm,amscd,dsfont,mathrsfs,mathtools,nicefrac, bm}
\usepackage{float,epsfig,color,xcolor,url}
\usepackage{subfigure}
\usepackage[ruled,vlined]{algorithm2e}
\usepackage{tablefootnote}
\usepackage{cleveref}
\graphicspath{{Plots/}}
\usepackage{tikz}
\usepackage{makecell}
\usepackage{multirow}
\usepackage{booktabs}
\usepackage[authoryear,round]{natbib}
\usepackage{lineno}

%% file: macro.tex
\newtheorem{theorem}{Theorem}[section]
\newtheorem{lemma}{Lemma}[section]

\newtheorem{proposition}{Proposition}[section]

\newtheorem{remark}{Remark}[section]

\newtheorem{definition}{Definition}[section]

\newcommand{\assign}{\coloneqq}

\newcommand{\reals}{\mathbb{R}}
\newcommand{\inner}[2]{\langle{#1},{#2}\rangle}

\newcommand{\argmin}{\operatorname{argmin}}

\newcommand{\dist}{\operatorname{dist}}

\newcommand{\D}{\operatorname{D}}
\renewcommand{\d}{\operatorname{d}}

\newcommand{\cO}{\mathcal{O}}
\newcommand{\cX}{\mathcal{X}}
\newcommand{\cK}{\mathcal{K}}

\newcommand{\cH}{\mathcal{H}}

\newcommand{\ee}{\mathrm{e}}
\newcommand{\plgd}{\texttt{PolyakGD}}
\newcommand{\polyak}{$\gamma$-\texttt{PolyakGD}}

\newcommand{\hd}{H\"older}

\newcommand{\mathsym}[1]{{}}
\newcommand{\unicode}[1]{{}}

\usepackage{setspace}

%% file: abstract.tex
\begin{abstract}

In this paper, we revisit a classical adaptive stepsize strategy for gradient descent: the Polyak
stepsize~(\plgd), originally proposed in \cite{polyak1969minimization}. We study the
convergence behavior of {\plgd} from two perspectives: tight worst-case analysis and universality across function classes. As our first main result, we establish the tightness of the known convergence rates of {\plgd} by explicitly constructing worst-case functions. In particular, we show that the $\cO((1-\frac{1}{\kappa})^K)$ rate for smooth strongly convex functions and the $\cO(1/K)$ rate for smooth convex functions are both tight. Moreover, we theoretically show that {\plgd} automatically exploits floating-point
errors to escape the worst-case behavior. Our second main result provides new convergence guarantees for {\plgd} under both
{\hd} smoothness and {\hd} growth conditions. These findings show that the
Polyak stepsize is universal, automatically adapting to various function classes without requiring prior knowledge of problem parameters.

\end{abstract}

%% file: sec_intro.tex
\section{Introduction}
The Polyak stepsize is a stepsize schedule for (sub)gradient descent in convex optimization, originally proposed by Boris T. Polyak~\citep{polyak1969minimization,polyak1987introduction}. It was initially designed for the subgradient method in nonsmooth optimization, and has recently been extended to more general settings~\citep{hazan2019revisiting,loizou2021stochastic,wang2023generalized,deng2024uniformly, devanathan2024polyak}. The Polyak stepsize requires knowledge of the optimal function value $f^\star$ and is defined as
\begin{equation}\label{eq:polyak stepsize}
    \alpha_k = \frac{f(x^k) - f^\star}{\|\nabla f(x^k)\|^2}, \ k = 0, 1, \ldots, K-1.
\end{equation}
{\plgd} automatically adjusts the stepsize according to the position of the current iterate $x^k$ on the optimization landscape. The resulting stepsizes
are typically much larger than standard constant or diminishing stepsizes, making {\plgd}~more aggressive and often yielding superior practical performance
\citep{zamani2024exact}. Consequently, the Polyak stepsize is widely used in settings
where the optimal function value $f^\star$ is known, such as convex feasibility
problems and over-parameterized machine learning models
\citep{boyd2013subgradient,loizou2021stochastic}. Even when $f^\star$ is not known
a priori, {\plgd} remains highly effective, provided that a reasonable estimate of the optimal value is available~\citep{polyak1969minimization,hazan2019revisiting,zheng2024adaptive}.

From a theoretical perspective, the convergence of the Polyak stepsize has been
extensively studied in nonsmooth convex optimization.~\cite{polyak1987introduction}
proved that the subgradient method with the Polyak stepsize satisfies
$\lim_{k \to \infty} \sqrt{k}\,(f(x^k) - f^\star) = 0$. Moreover, the convergence rate
of the best-iterate\footnote{By \emph{best-iterate}, we refer to $\min_{1 \le k \le K} (f(x^k) - f^\star)$; by \emph{last-iterate}, we refer to
$f(x^K) - f^\star$.} is $\cO(1/\sqrt{K})$ (see, e.g.,~\cite{boyd2013subgradient}).
Beyond best-iterate guarantees,~\cite{zamani2024exact} analyzed the last-iterate
behavior of the Polyak stepsize, establishing an $\cO(1/\sqrt[4]{K})$ convergence rate
and providing a worst-case function to match this rate. Their analysis uses the methodology based on performance
estimation problem (PEP)~\citep{drori2014performance,taylor2017smooth}, building on earlier work
\citep{zamani2025exact}\footnote{\cite{zamani2024exact,zamani2025exact} employ PEP to
obtain numerical dual multipliers that guide the derivation of analytical expressions
for their proofs.}. Additionally,~\cite{criscitiello2023curvature} provided a worst-case function for subgradient method with the Polyak stepsize on Riemannian manifolds. More recently,~\cite{orabona2025new} showed that the Polyak stepsize can be interpreted
as applying the subgradient method to a specific surrogate function, providing
an alternative justification for the necessity of acquiring $f^\star$. In contrast to the well-studied nonsmooth setting, the smooth setting remains largely unexplored. The primary reference is the
note~\cite{hazan2019revisiting} by Hanzan and Kakade, who established convergence guarantees for {\plgd} on both smooth convex and smooth strongly convex objectives. In the
smooth convex setting, they proved an $\cO(1/K)$ convergence rate for the best-iterate,
matching the rate of constant-stepsize gradient descent.

Although {\plgd} is adaptive and often exhibits superior practical performance, the existing $\cO(1/K)$ convergence rate does not exceed that of gradient descent with a constant stepsize. It naturally motivates the following question:
\begin{center}
    \textit{Is the existing upper bound for the Polyak stepsize tight?}
\end{center}
Two main approaches to tight convergence analysis have been developed in the literature. One is to construct worst-case functions. For example,~\cite{drori2014performance,teboulle2023elementary} designed the Huber loss to justify the tightness of the $\cO(1/K)$ rate for constant-stepsize gradient descent. The other one is PEP, which works well for designing a fixed stepsize schedule~\citep{altschuler2025acceleration,grimmer2024provably,rubbens2025performance}. To our best knowledge, there are only two papers on the tight convergence of the Polyak stepsize, motivated from the PEP-based methodology. They focus on the strongly convex case and study the variants of the original Polyak stepsize~\citep{barre2020complexity,huang2024analytic}. PEP remains applicable in the strongly convex setting, as the contraction ratio between consecutive iterations can be explicitly analyzed. However, when the objective function is not strongly convex, PEP does not readily extend to adaptive stepsizes~\citep{goujaud2023fundamental,goujaud2024pepit}. 

Our paper provides a two-fold answer to the above question. On the one hand, we show that the known convergence rates for {\plgd} are indeed tight by constructing worst-case functions. On the other hand, we demonstrate that under floating-point arithmetic, {\plgd} can automatically leverage numerical error to escape this worst-case behavior. This interesting phenomenon explains the superior empirical performance of {\plgd} in practice. Theoretically, the strong performance of {\plgd} is attributed to its inherent adaptivity, a crucial property for modern optimization methods~\citep{malitsky2020adaptive,gao2024gradient,li2025simple,he2025history}.~\cite{hazan2019revisiting} showed that the Polyak stepsize is simultaneously adaptive
to both smooth and nonsmooth regimes, encompassing general convex as well as strongly
convex objectives.~\cite{mishkin2024directional} further established that it adapts to
directional smoothness. More recently,~\cite{orabona2025new} demonstrated that the
Polyak stepsize also adapts to the curvature of a surrogate function induced by the
original objective. Inspired by these results, we further advance this line of research and provide a comprehensive answer to the following question: 
\begin{center}
    \textit{Can the adaptivity of the Polyak stepsize be further extended?}
\end{center}
\paragraph{Contributions} This work answers the above two questions. To address the first question, we construct a simple two-dimensional quadratic function. Our proof strategy relies on choosing a specific initial point such that the Polyak stepsize reduces to a constant stepsize along the trajectory in the strongly convex case. Subsequently, we derive the worst-case functions for the general convex and {\hd} smooth cases by adapting this quadratic function. Based on this construction, we establish the tightness of the linear convergence rate for strongly convex functions, the $\cO(1/K)$ rate for general convex functions, and the $\cO(1/K^{(\nu + 1)/2})$ rate for $\nu$-{\hd} smooth objectives. We then show that under floating-point arithmetic, {\plgd} escapes this worst-case behavior.

We manage to address the second question by demonstrating that the Polyak stepsize is universal across various function classes. We establish that {\plgd} simultaneously adapts to both {\hd} smoothness and {\hd} growth conditions (Theorem \ref{thm:Holder smooth growth}). In particular, when only the {\hd} growth condition holds, {\plgd} achieves the optimal convergence rate; when the function is {\hd} smooth, it matches the rate of
the universal gradient method of~\cite{nesterov2015universal}. Furthermore, we show
that the Polyak stepsize adapts to the recently proposed intrinsic global curvature
bound~\cite{nesterov2025universal}, and that our universal analysis extends to the
stochastic setting under the interpolation condition. A summary of these results is
provided in Table~\ref{table:complexity}.

\begin{table}[!h]
\centering
\caption{Upper/lower bounds of {\plgd}.}
\resizebox{0.92\textwidth}{!}{
\begin{tabular}{ccc}
    \toprule
    Function class 
        & Upper bound (rate) 
        & Lower bound (rate) \\
    \midrule

    \multirow{2}{*}{$L$-smooth}
        & $\mathcal{O}(1/K)$ 
        & $\Omega(1/K)$ \\
        & {\small \cite{hazan2019revisiting}} 
        & {\small Theorem \ref{thm:smooth-cvx-lb}} \\[2mm]

    \multirow{2}{*}{$L$-smooth $\mu$-strongly convex}
        & $\mathcal{O}((1 - \frac{1}{\kappa})^K)$ 
        & $\Omega((1 - \frac{1}{\kappa})^K)$ \\
        & {\small \cite{hazan2019revisiting}} 
        & {\small Theorem \ref{thm:lb-strongcvx}} \\[2mm]

    \midrule

    \multirow{2}{*}{$\nu$-H{\"o}lder smooth}
        & $\mathcal{O}(K^{-(\nu+1)/2})$ 
        & $\Omega(K^{-(\nu+1)/2})$ \\ 
        & {\small \cite{orabona2025new}} 
        & Theorem \ref{thm:Holder smooth-cvx-lb} \\[2mm]

    \multirow{2}{*}{$r$-H{\"o}lder growth}
        & $\mathcal{O}(K^{-\, r/(2(r-1))})$ 
        & $\Omega(K^{-\, r/(2(r-1))})$ \\
        & {\small Theorem \ref{thm:Holder smooth growth}} 
        &  \cite{nemirovskii1985optimal} \\[2mm]

    \multirow{2}{*}{$\nu$-H{\"o}lder smooth + $r$-H{\"o}lder growth}
        & $\mathcal{O}(K^{-\, r(\nu+1)/(2(r-\nu-1))})$ 
        & \multirow{2}{*}{---} \\  
        & {\small Theorem \ref{thm:Holder smooth growth}} 
        &  \\[2mm]               

    \bottomrule
\end{tabular}
}
\label{table:complexity}
\end{table}

\section{Preliminaries}\label{section:preliminary}
We consider the following (un)constrained optimization problem
\begin{equation}\label{eq:main}
        \min_{x \in \reals^n} \ f(x), 
\end{equation}
where the objective function $f: \reals^n \rightarrow \reals$ is convex and differentiable\footnote{Our analysis can be extended to the projected (sub)gradient method for the constrained setting $\min_{x \in \cX} f(x)$ with a closed convex constraint set $\cX$ and $f$ is nonsmooth.}. We assume that the optimal solution set $\cX^\star$ is nonempty and the optimal function value $f^\star = f(x^\star)$ is finite. Besides, we assume that $f^\star$ is known. Addressing the scenario where $f^\star$ is unknown is beyond the scope of this paper, and the techniques developed in \cite{polyak1969minimization,hazan2019revisiting} can be applied to estimate $f^\star$ at the cost of an additional $\log(1/\varepsilon)$ factor in the overall complexity.

\vspace{10pt}
\begin{algorithm}[H]
{\textbf{input} initial point $x^0$, optimal value $f^\star$, scaling parameter $\gamma \in (0, 2]$}

\For{$k = 0, 1, \dots$}{
$\alpha_k = \frac{f(x^k) - f^\star}{\|\nabla f(x^k)\|^2}$\\
$x^{k + 1} = x^k - \gamma \cdot \alpha_k \nabla f (x^k)$\\
}
\textbf{output}
\caption{\polyak \label{alg:polyakgd}}
\end{algorithm}
\vspace{10pt}

Throughout the paper, we study the convergence behavior of gradient descent with a scaled Polyak stepsize (Algorithm \ref{alg:polyakgd}), where $\gamma \in (0, 2]$ is a scaling parameter; we refer to this update as \polyak. Clearly, setting $\gamma = 1$ recovers the standard Polyak stepsize gradient descent. For the upper curvature bound, we adopt the {\hd} smoothness condition:
\begin{definition} \label{def:smooth}
   A differentiable function $f$ is $(\nu,L_\nu)$-{\hd} smooth if it satisfies
   \begin{equation}\label{eq:Holder smooth}
    \|\nabla f(x) - \nabla f(y)\| \le L_\nu \|x - y\|^\nu, \ \nu \in (0,1], \ L_\nu > 0, \text{ for all } x, y \in \reals^n.
\end{equation}
\end{definition} 
Specifically, the case $\nu = 1$ corresponds to the standard $L$-smoothness condition. The $(\nu, L_\nu)$-{\hd} smoothness condition implies the following two inequalities:
\begin{equation}\label{eq:Holder smooth inequalities}
    \begin{aligned}
        f(y) &\le f(x) + \inner{\nabla f(x)}{y - x} + \frac{L_\nu}{\nu+1}\|y - x\|^{\nu+1}, \ \text{ for all } x, y \in \reals^n, \\
       \|\nabla f(x)\|^{\frac{\nu + 1}{\nu}} &\le \frac{\nu + 1}{\nu} L_\nu^{\frac{1}{\nu}} (f(x) - f^\star), \ \text{ for all } x \in \reals^n,   
    \end{aligned}
\end{equation}
with proofs provided in \cite{orabona2023normalized}. For the lower curvature bound, we adopt the following growth condition:
\begin{definition} \label{def:growth}
   A continuous function $f$ satisfies the $(r, \rho_r)$-{\hd} growth condition on a nonempty set $\cK \subseteq \reals^n$ if
    \begin{equation}\label{eq:Holder growth}
        f(x) - f^\star \ge \rho_r \dist(x, \cX^\star)^r, \ r \ge 1, \ \rho_r > 0, \ \text{ for all } x \in \cK,
    \end{equation}
    where $\dist(x, \cX^\star) \coloneqq\inf_{x^\star \in \cX^\star} \ \|x - x^\star\|$.
\end{definition}
Note that the first inequality in Equation \eqref{eq:Holder smooth inequalities} implies 
\begin{align*}
    f(x) - f^\star \le \frac{L_\nu}{\nu+1}\dist(x, \cX^\star)^{\nu+1}, \ \text{ for all } x \in \reals^n.
\end{align*}
Combining this with the $(r, \rho_r)$-{\hd} growth condition on the set $\cK$, we have
\begin{align*}
   \rho_r \dist(x, \cX^\star)^r \le \frac{L_\nu}{\nu+1}\dist(x, \cX^\star)^{\nu+1}, \ \text{ for all } x \in \cK. 
\end{align*}
This implies that 
\begin{equation}\label{eq:relation of nu and rho}
    \frac{\rho_r(\nu+1)}{L_\nu} \le \dist(x, \cX^\star)^{(\nu+1) - r} , \ \text{ for all } x \in \cK.
\end{equation}
Consequently, several subtle points deserve attention. First, we require that $\cX^\star \subseteq \cK$, otherwise the {\hd} growth condition does not seem intuitive. As a result, we exclude the case $r < \nu + 1$, as choosing $x$ sufficiently close to $\cX^\star$ would lead to a contradiction with the inequality \eqref{eq:relation of nu and rho}. Moreover, we require that $\sup_{x \in \cK} \dist(x, \cX^\star) < \infty$. Otherwise, the case $r > \nu + 1$ would also yield a contradiction by taking $x$ such that $\dist(x, \cX^\star) \rightarrow \infty$ (see Section 1.2 of \cite{roulet2017sharpness} for a similar argument). Fortunately, as discussed at the beginning of Section \ref{section:universality}, {\polyak} (Algorithm~\ref{alg:polyakgd}) naturally ensures the existence of such a proper set $\cK$ due to the Fej\'{e}r monotonicity.

%% file: sec_lb.tex
\section{Tight Convergence of {{\polyak}}}\label{section:tight convergence}
In this section, we study the tightness of the convergence rates of {\polyak} for various function classes. Taking the class of convex $L$-smooth functions as an example, the formal problem is: given a total iteration count $K$, can we find a convex and $L$-smooth function $f$ and an initial point $x^0$ such that the trajectory generated by {\polyak} satisfies
\begin{align*}
    \min_{1 \le k \le K} \ \frac{f(x^k) - f^\star}{L \|x^0 - x^\star\|^2} \ge \Omega\Big(\frac{1}{K}\Big)?
\end{align*} 
Consequently, the main task is to construct a worst-case function to answer the above question. For instance, \cite{drori2014performance,teboulle2023elementary} utilized the Huber loss to establish the tightness of the $\cO(1/K)$ rate for constant stepsize gradient descent. The intuition is that the small gradient magnitude of the Huber loss forces a constant stepsize gradient descent to take conservatively small steps, thereby justifying the tightness. However, this strategy fails due to the adaptivity of {\polyak}, as demonstrated below.
\begin{proposition}\label{prop:Huber loss}
    Let $K \ge 1$ be the total iteration count. Consider the Huber loss defined by
    \begin{align*}
        \mathcal{H}_K(x) \coloneqq
            \begin{cases}
            \frac{1}{2K+1}\|x\| - \frac{1}{2(2K+1)^2}, & \|x\| > \frac{1}{2K+1},\\
            \tfrac{1}{2}\|x\|^2, & \|x\| \le \frac{1}{2K+1}.
            \end{cases}
    \end{align*}
    Then, for any initial point $x^0 \in \reals^n$, $1$-\texttt{PolyakGD} converges to $x^\star = 0$ in two steps.
\end{proposition}

Therefore, our strategy is to construct a function such that the Polyak stepsize reduces to a constant stepsize along the trajectory. We remark that the first two theorems (Theorems \ref{thm:lb-strongcvx} and \ref{thm:smooth-cvx-lb}) allow for a wider range of $\gamma \in (0,4)$, which facilitates the analysis for the worst-case function under {\hd} smoothness (Theorem \ref{thm:Holder smooth-cvx-lb}).

\subsection{Worst-case quadratic functions for {\polyak}}
We start with the case where the objective function $f$ is $\mu$-strongly convex ($\mu > 0$) and $L$-smooth.

\begin{theorem} \label{thm:lb-strongcvx} 
Given the parameter $\gamma \in (0, 4)$, let $\kappa > \max\{4/ \gamma - 1, \gamma / (4-\gamma)\}$ be the condition number. Consider the following two-dimensional quadratic optimization problem 
  \begin{align*}
      \min_{x \in \reals^2} \ q (x) \coloneqq
  \frac{\kappa}{2} x_{(1)}^2 + \frac{1}{2} x_{(2)}^2.
  \end{align*} Suppose the initial point $x^0 = (x^0_{(1)}, x^0_{(2)})$ satisfies
  \begin{equation}\label{eq:initial point of lower bound function}
  (x_{(2)}^0)^2 = \frac{4 \kappa^2 - \gamma \kappa (\kappa + 1)}{\gamma ( \kappa + 1) - 4}  (x_{(1)}^0)^2.
 \end{equation}
Then {\polyak} generates a trajectory of $\{ x^k \}$ such that
  \[ \Bigg| \frac{x_{(1)}^{k + 1}}{x_{(1)}^k} \Bigg| = \left| \frac{x_{(2)}^{k + 1}}{x_{(2)}^k}
     \right| = \frac{\kappa - 1}{\kappa + 1}, \quad \text{and} \quad\ q(x^{k}) - q^\star = \Big(\frac{\kappa - 1}{\kappa + 1} \Big)^{2k} (q(x^0) - q^\star). \]
\end{theorem}
\begin{proof}
    First, note that the initial point \eqref{eq:initial point of lower bound function} is well-defined provided that $\kappa > \max\{4/ \gamma - 1, \gamma / (4-\gamma)\}$ for any $\gamma \in (0,4)$. We now prove by mathematical induction that the trajectory $\{x^k\}$ generated by {\polyak} satisfies the following relationship:
    \begin{equation}\label{eq:ratio of x1 and x2}
        (x_{(2)}^k)^2 = \frac{4 \kappa^2 - \gamma \kappa (\kappa + 1)}{\gamma ( \kappa + 1) - 4}  (x_{(1)}^k)^2, \ \text{ for all } k \ge 0.
    \end{equation}
    
\textsf{Base case.} The basic case $k = 0$ holds trivially due to the condition \eqref{eq:initial point of lower bound function}. 

\textsf{Inductive step.} Assuming the case holds for $k$, we now consider the case $k + 1$. Since the optimal function value $q^\star = 0$, the Polyak stepsize at iteration $k$ is given by
    \begin{align*}
       \alpha_k = \frac{q(x^k)}{\| \nabla q(x^k) \|^2} = \frac{\frac{\kappa}{2} (x_{(1)}^k)^2 + \frac{1}{2} (x_{(2)}^k)^2}{\kappa^2 (x_{(1)}^k)^2 + (x_{(2)}^k)^2} = \frac{\frac{\kappa}{2} (x_{(1)}^k)^2 + \frac{1}{2} \frac{4 \kappa^2 - \gamma \kappa (\kappa + 1)}{\gamma ( \kappa + 1) - 4}(x_{(1)}^k)^2}{\kappa^2 (x_{(1)}^k)^2 + \frac{4 \kappa^2 - \gamma \kappa (\kappa + 1)}{\gamma ( \kappa + 1) - 4}(x_{(1)}^k)^2} 
       = \frac{1}{\gamma }\frac{2}{\kappa
       + 1}.
    \end{align*}
    Consequently, the next iterate $x^{k+1}$ is
    \begin{align*}
        x^{k+1}_{(1)} &= x^k_{(1)} - \gamma \cdot \frac{1}{\gamma}\frac{2}{\kappa + 1} \kappa x_{(1)}^k = - \frac{\kappa - 1}{\kappa +
      1} x^k_{(1)}, \\
      x^{k+1}_{(2)} &= x^k_{(2)} - \gamma \cdot \frac{1}{\gamma} \frac{2}{\kappa + 1} x^k_{(2)} = \frac{\kappa - 1}{\kappa +
      1} x^k_{(2)},
    \end{align*}
    which implies that
    \begin{equation}\label{eq:ratio between two iterations}
       \left| \frac{x_{(1)}^{k+1}}{x_{(1)}^k} \right| = \left| \frac{x_{(2)}^{k+1}}{x_{(2)}^k}
     \right| = \frac{\kappa - 1}{\kappa + 1}.
    \end{equation}
The iterate $x^{k+1}$ thus satisfies the relationship \eqref{eq:ratio of x1 and x2}, which follows from the inductive hypothesis at iteration $k$. Therefore, by the principle of mathematical induction, the relationship \eqref{eq:ratio between two iterations} also holds for each iteration $k$, and the iterates $\{x^k\}$ converge to $x^\star$ along a fixed trajectory (Figure \ref{fig:trajectory}).

\begin{figure*}
\centering
\includegraphics[scale=0.45]{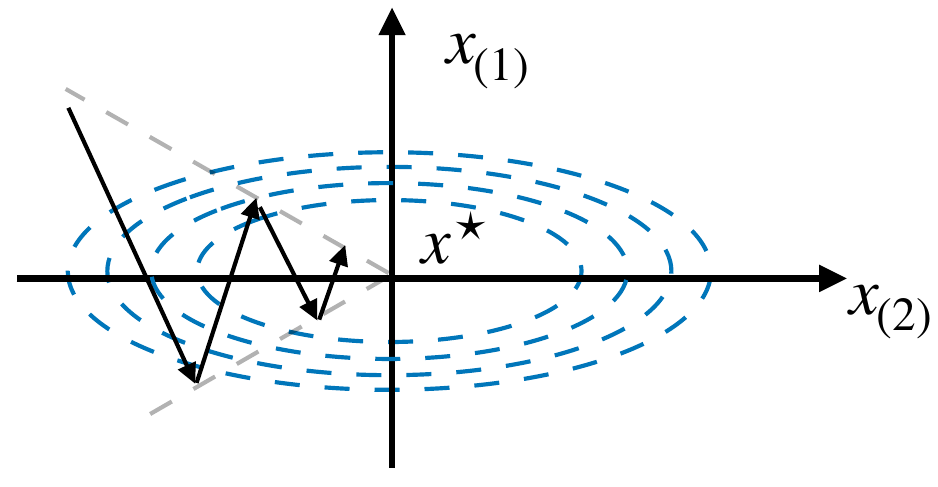}
\caption{Behavior of {\polyak} on the worst-case function \label{fig:trajectory}}
\end{figure*}

Finally, by calculating the function value at iteration $k+1$ with condition \eqref{eq:ratio between two iterations}, we arrive at the following recurrence relation:
    \begin{align*}
        q(x^{k}) = \frac{\kappa}{2} (x_{(1)}^{k+1})^2 + \frac{1}{2} (x_{(2)}^{k+1})^2
       = \left(\frac{\kappa - 1}{\kappa + 1}\right)^2\left(\frac{\kappa}{2}
       (x_{(1)}^k)^2 + \frac{1}{2} (x_{(2)}^k)^2\right) = \left(\frac{\kappa - 1}{\kappa + 1}\right)^2 q(x^{k-1}).
    \end{align*}
    This completes the proof by substituting $q^\star = 0$.
\end{proof}


We now turn to constructing the worst-case function for the class of $L$-smooth convex functions. Our strategy is to adapt the two-dimensional quadratic function $q $ in Theorem \ref{thm:lb-strongcvx} via a strongly-convex to convex reduction. Specifically, by setting the condition number $\kappa$ proportional to $K$, the linear convergence (in terms of $\kappa$) translates to sublinear convergence (in terms of $K$). However, this approach makes the gradient Lipschitz constant of $q $ dependent on $K$, implying that $q $ is not $L$-smooth. Crucially, we observe that although the Polyak stepsize itself scales inversely with the objective function, the resulting trajectory of {\polyak} remains invariant to such scaling. Define a scaled quadratic function
\begin{equation}\label{eq:scaled lower bound function}
    \Tilde{q}(x) \coloneqq \frac{L}{\kappa} q(x),
\end{equation}
where $L > 0$ is a constant. Then $\Tilde{q} $ is $L$-smooth by definition, and thus $\Tilde{q} $ is an $L$-smooth convex function. These two components lead to the following theorem.
   
\begin{theorem} \label{thm:smooth-cvx-lb}
  Given the parameter $\gamma \in (0, 4)$, let $K \geq 1$ be the total iteration count and $L > 0$ be a constant. Then there exists an $L$-smooth convex function $\Tilde{q} $ and an initial point $x^0$ such that the trajectory generated by {\polyak} satisfies
  \begin{align*}
      \min_{1 \le k \le K} \ \frac{\Tilde{q}(x^{k}) - \Tilde{q}^\star}{L\|x^0 - x^{\star}\|^2} \geq \frac{\gamma}{2\ee^{2\gamma}((4-\gamma)K+\gamma)},
  \end{align*}
  where $\ee$ denotes Euler's number.
\end{theorem}
\begin{proof}
    We begin with the same setting as in Theorem \ref{thm:lb-strongcvx}. Suppose the condition number $\kappa$ satisfies $\kappa > \max\{4/ \gamma - 1, \gamma / (4-\gamma)\}$, a requirement that will be justified by our subsequent choice of $\kappa$. Without loss of generality, we choose the initial point $x^0$ on the unit sphere, from which it holds that:
    \begin{align*}
        1 = (x_{(1)}^0)^2 + (x_{(2)}^0)^2 = (x_{(1)}^0)^2 + \frac{4 - \gamma ( \kappa + 1) }{\gamma \kappa (\kappa + 1) - 4 \kappa^2 }  (x_{(1)}^0)^2.
    \end{align*}
    Then the initial point $x^0$ follows as
    \begin{align}\label{eq:initial point in the convex lower bound function}
       (x_{(1)}^0)^2 & ={} \frac{\gamma (\kappa + 1) - 4}{(4 - \gamma)(\kappa^2 - 1) } = -\frac{1}{\kappa^2 - 1} + \frac{\gamma \kappa}{(4 - \gamma)(\kappa^2 - 1)}, \\
       \ (x_{(2)}^0)^2 & ={} 1 - (x_{(1)}^0)^2 = \frac{\kappa^2}{\kappa^2 - 1} - \frac{\gamma \kappa}{(4 - \gamma)(\kappa^2 - 1)}, 
    \end{align}
    giving
    \begin{align*}
       q(x^0) = \frac{\kappa}{2}(x_{(1)}^0)^2 + \frac{1}{2}(x_{(2)}^0)^2  = \frac{2\kappa}{(4-\gamma)(\kappa + 1)}. 
    \end{align*}
    Using the results in Theorem \ref{thm:lb-strongcvx} yields
    \begin{align*}
       q(x^{K}) - q^\star = \left(\frac{\kappa - 1}{\kappa + 1}\right)^{2 K} q(x^0) = \left(\frac{\kappa - 1}{\kappa + 1}\right)^{2 K}\frac{2\kappa}{(4-\gamma)(\kappa + 1)}.  
    \end{align*}
    Consequently, plugging in the definition of $\tilde{q}  $, we have
    \begin{align*}
       \Tilde{q}(x^{K}) - \Tilde{q}^\star = \left(\frac{\kappa - 1}{\kappa + 1}\right)^{2 K} \Tilde{q}(x^0) = \left(\frac{\kappa - 1}{\kappa + 1}\right)^{2 K} \frac{L}{\kappa}q(x^0) = \left(\frac{\kappa - 1}{\kappa + 1}\right)^{2 K}\frac{2L}{(4-\gamma)(\kappa + 1)}.  
    \end{align*}
    The above equality implies that the function value gap is monotone in this case, and it follows that
    \begin{align*}
       \Tilde{q}(x^{K}) - \Tilde{q}^\star = \min_{1 \le k \le K} \ \Tilde{q}(x^{k}) - \Tilde{q}^\star = \left(\frac{\kappa - 1}{\kappa + 1}\right)^{2 K}\frac{2L}{(4-\gamma)(\kappa + 1)}.  
    \end{align*}
    Now we choose $\kappa = 4 K / \gamma + \gamma/(4-\gamma) > \max\{4/ \gamma - 1, \gamma / (4-\gamma)\}$. Note that $x^\star$ is the origin and $\|x^0\| = 1$. We have
    \begin{align*}
        \min_{1 \le k \le K} \ \frac{\Tilde{q}(x^{k}) - \Tilde{q}^\star}{L\|x^0 - x^\star\|^2} = \left(\frac{8K-2(1 + K)\gamma + \gamma^2}{2K(4-\gamma) + 2\gamma}\right)^{2K}\frac{\gamma}{2((4-\gamma)K+\gamma)}.
    \end{align*}
    Therefore, the proof is completed by using technical lemma \ref{lemma:lb e-gamma}.
\end{proof}

\begin{remark}
    Our reduction argument is also applicable to the BB stepsize on quadratic functions due to its scale invariance. By leveraging the tight convergence result and the worst-case quadratic function from \cite{li2021faster}, a similar $\Omega(1/K)$ result can be established for the BB stepsize.
\end{remark}

We extend the above analysis to the {\hd} smooth setting in the following theorem. The proof is deferred to the appendix. Given $\nu \in (0,1]$, we construct a worst-case function $q_{\nu}(x) \coloneqq(\Tilde{q}(x))^{\frac{\nu+1}{2}}$, which can be shown to be $(\nu,L_\nu)$-{\hd} smooth for some constant $L_\nu > 0$. 
\begin{theorem} \label{thm:Holder smooth-cvx-lb}
  Given the parameter $\gamma \in (0, 2)$ and $\nu \in (0,1]$, let $K \geq 1$ be the total iteration count and $L_{\nu} > 0$ be a constant. Then there exists a $(\nu,L_\nu)$-{\hd} smooth convex function $q_\nu $ and an initial point $x^0$ such that the trajectory generated by {\polyak} satisfies
  \begin{align*}
       \min_{1 \le k \le K} \ \frac{q_\nu(x^k) - q_{\nu}(x^\star)}{L_\nu\|x^0 - x^\star\|^{\nu + 1}} \ge \frac{2^{\nu - 1}\gamma^{\frac{\nu + 1}{2}}}{\ee^{2\gamma} (\nu + 1)((2(\nu + 1) - \gamma)K + \gamma)^{\frac{\nu + 1}{2}}},
    \end{align*} 
  where $\ee$ denotes Euler's number.
\end{theorem}
\begin{remark}
    Setting $\nu \rightarrow  0$, the above result reduces to the lower bound $\Omega(1/\sqrt{K})$. We remark that this result does not contradict the $\Omega(1/\sqrt[4]{K})$ bound established in \cite{zamani2024exact}. Our tightness result is established for the best iterate, while theirs applies to the last iterate, which is generally a more challenging criterion. Consequently, the worst-case function constructed in \cite{zamani2024exact} is more complicated. This distinction reveals the significant gap between the best iterate and the last iterate under bounded gradient condition, highlighting the necessity of analyzing both cases.
\end{remark}
Finally, we conclude this subsection by addressing the tightness of the gradient norm convergence. Since our worst-case function $\Tilde{q} $ is quadratic, it naturally allows us to establish this result as a byproduct.

\begin{theorem} \label{thm:smooth-cvx-lb gradient norm}
  Given the parameter $\gamma \in (0, 2)$ and $\nu \in (0,1]$, let $K \geq 1$ be the total iteration count and $L_{\nu} > 0$ be a constant. Then there exists a $(\nu,L_\nu)$-{\hd} smooth convex function $q_\nu $ and an initial point $x^0$ such that the trajectory generated by {\polyak} satisfies
  \begin{align*}
        \min_{1 \le k \le K} \ \frac{\|\nabla q_\nu(x^k)\|}{L_\nu \|x^0 - x^\star\|^{\frac{\nu + 1}{2}}} \ge \frac{2^{\frac{\nu-1}{2}}\gamma^{\frac{\nu+1}{2}}
(\nu+1)^{\frac{\nu-1}{2}}\sqrt{4-\gamma}}
{\ee^{\gamma\nu}\sqrt{4(4-\gamma)+\gamma^2}(2(\nu+1)-\gamma)^{\frac{\nu}{2}}K^{\frac{\nu}{2}}},
    \end{align*}
  where $\ee$ denotes Euler's number.
\end{theorem}


\begin{remark}
    Establishing the tight gradient norm convergence of constant stepsize gradient descent involves subtle distinctions. Setting $\nu = 1$, the above result reduces to the lower bound $\Omega(1/\sqrt{K})$. However, \cite{nesterov2018lectures,diakonikolas2021potential,teboulle2023elementary} established that small constant stepsizes achieve an upper bound of $\cO(1/K)$, a result shown to be tight by \cite{kim2018generalizing} via PEP. Using a slightly large constant stepsize slows the convergence to $\cO(1/\sqrt{K})$, as shown in Proposition 4.4 in \cite{rotaru2022tight} and Theorem 2.1 in \cite{rotaru2024exact}, a bound that is also tight according to PEP. Therefore, our result suggests that the Polyak stepsize behaves like a ``large constant stepsize" in some sense.
\end{remark}



\subsection{Polyak stepsize escapes the worst case under floating-point arithmetic}
Our tightness convergence results show that {\plgd} reduces to gradient descent with a constant stepsize on the worst-case quadratic function. While this conclusion holds under exact arithmetic, it may not persist in practical implementations due to floating-point errors. Indeed, when implementing the Polyak stepsize on the constructed worst-case quadratic function, the numerical behavior differs from the theoretical results. Figure~\ref{fig:instability} plots the stepsize sequence $\{\alpha_k\}$ and the corresponding function value gaps $\{q(x^k) - q^{\star} \}$ generated by $1$-\texttt{PolyakGD}. As predicted by our theory, $1$-\texttt{PolyakGD} initially behaves like constant stepsize gradient descent. Yet after a few iterations, it departs from the theoretical trajectory, producing larger stepsizes and converging significantly faster. This suggests that under inexact arithmetic, {\polyak} (with $\gamma \in (0,2)$) can automatically exploit floating-point errors to escape the worst-case behavior, which may explain the superior practical performance of the Polyak stepsize. This section formally quantifies this intuition. The techniques developed in this section have also been applied to analyze the hypergradient descent heuristic recently \citep{chu2025gradient}.

\begin{figure}[h]
\centering\includegraphics[width=0.5\linewidth]{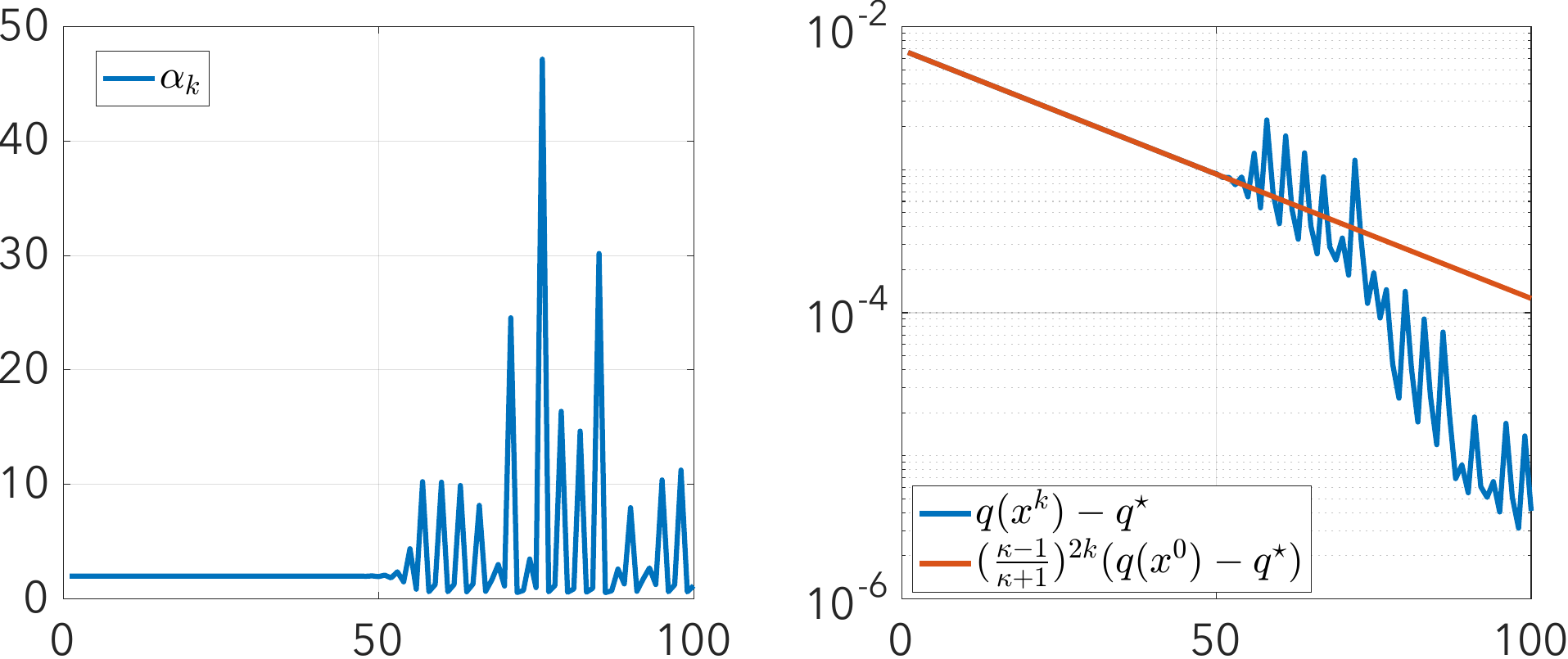}
    \caption{Instability in the presence of floating-point error allows {\plgd} to escape the worst-case}
    \label{fig:instability}
\end{figure}

\paragraph{Dynamical system formulation of {\polyak}.} To simplify notation, for this subsection only, we will overload the notation by $z \coloneqq (x, y) = (x_{(1)}, x_{(2)})$. In this new notation, the worst-case function is $q (z) \assign \tfrac{\kappa}{2} x^2 + y^2$. Since $q $ is homogeneous with optimum $z^\star = (0, 0)$, it suffices to consider $\frac{z}{\|z\|}$ to analyze the algorithm trajectory. Our analysis starts by rewriting {\polyak} as a nonlinear dynamical system: define the state variable $s = ( \tfrac{x}{\| z \|}, \tfrac{y}{\| z \|},
\alpha)$. Letting $z^+ = z - \alpha \nabla f(z)$, we have 
\begin{equation} \label{eqn:dynamic}
	s^+ = \left(\begin{array}{c}
     \tfrac{x^+}{\| z^+ \|}\\
     \tfrac{y^+}{\| z^+ \|}\\
     \alpha^+
   \end{array}\right) =\mathcal{F}_{\gamma} (s) \assign \left(\begin{array}{c}
     \tfrac{(1 - \alpha \kappa) x}{\sqrt{(1 - \alpha \kappa)^2 x^2 + (1 -
     \alpha)^2 y^2}}\\
     \tfrac{(1 - \alpha) y}{\sqrt{(1 - \alpha \kappa)^2 x^2 + (1 - \alpha)^2
     y^2}}\\
     \tfrac{\gamma}{2} \tfrac{\kappa x^2 + y^2}{(1 - \alpha \kappa)^2 \kappa^2
     x^2 + (1 - \alpha)^2 y^2}
   \end{array}\right) .
\end{equation}
 
According to Theorem~\ref{thm:lb-strongcvx}, when the initial point of {\polyak} satisfies
Equation~\eqref{eq:initial point of lower bound function}, the algorithmic trajectory
oscillates between two rays. In the language of dynamical systems, there exist two
states $s_1, s_2$ such that
\[
    s_1 = \mathcal{F}_{\gamma}(s_2), \qquad 
    s_2 = \mathcal{F}_{\gamma}(s_1)
        = \mathcal{F}_{\gamma}(\mathcal{F}_{\gamma}(s_2)),
\]
and $s_1$ (or equivalently $s_2$) is a period-2 orbit of $\mathcal{F}_{\gamma}$
\citep{alligood1997chaos}.  
To understand the local behavior of {\polyak} in the presence of floating-point
errors, it is natural to examine the stability of this nonlinear map around its
orbit. This reduces to analyzing the spectrum of the Jacobian of $\mathcal{F}_{\gamma} $. The proof of Proposition \ref{prop:dynamical-system} is assisted by \texttt{Mathematica} and the source code is available at \href{https://github.com/Gwzwpxz/polyak_proof}{this link}.

\begin{proposition} \label{prop:dynamical-system}
  Suppose $\gamma \in (0, 2]$ and $\kappa > \max\{4 \gamma^{- 1} - 1, 3 + 2 \sqrt{2}$\}. Then the nonlinear dynamical system \eqref{eqn:dynamic}
  has period-two orbits given by
  \[ \mathcal{S} = \left\{ \footnotesize (s_1, s_2) : s_1 = \left(\begin{smallmatrix}
       x\\
       y\\
       \tfrac{2}{\kappa + 1}
     \end{smallmatrix}\right), s_2 = \left(\begin{smallmatrix}
       -x\\
       y\\
       \tfrac{2}{\kappa + 1}
     \end{smallmatrix}\right), x^2 = \tfrac{\gamma (\kappa + 1) - 4}{(\gamma - 4)
     (\kappa^2 - 1)}, y^2 = \tfrac{\kappa (\gamma - 4 \kappa + \gamma
     \kappa)}{(\gamma - 4) (\kappa^2 - 1)} \right\} . \]
  Moreover, the Jacobian of the dynamical system is given by
  \[ \mathcal{J}_{\mathcal{F}_{\gamma}} (s) = \small{\left(\begin{array}{ccc}
       - \frac{(\alpha - 1)^2 y^2 (\alpha \kappa - 1)}{((x - \alpha \kappa
       x)^2 + (y - \alpha y)^2)^{3 / 2}} & \frac{(\alpha - 1)^2 x y (\alpha
       \kappa - 1)}{((x - \alpha \kappa x)^2 + (y - \alpha y)^2)^{3 / 2}} &
       \frac{(\alpha - 1) (\kappa - 1) x y^2}{((x - \alpha \kappa x)^2 + (y -
       \alpha y)^2)^{3 / 2}}\\
       \frac{(\alpha - 1) x y (\alpha \kappa - 1)^2}{((x - \alpha \kappa x)^2
       + (y - \alpha y)^2)^{3 / 2}} & - \frac{(\alpha - 1) x^2 (\alpha \kappa
       - 1)^2}{((x - \alpha \kappa x)^2 + (y - \alpha y)^2)^{3 / 2}} & -
       \frac{(\kappa - 1) x^2 y (\alpha \kappa - 1)}{((x - \alpha \kappa x)^2
       + (y - \alpha y)^2)^{3 / 2}}\\
       - \frac{(\alpha - 1)^2 \gamma (\kappa - 1) \kappa x y^2 (\alpha \kappa
       - 1)^2}{(\kappa^2 x^2 (\alpha \kappa - 1)^2 + (\alpha - 1)^2 y^2)^2} &
       \frac{(\alpha - 1)^2 \gamma (\kappa - 1) \kappa x^2 y (\alpha \kappa -
       1)^2}{(\kappa^2 x^2 (\alpha \kappa - 1)^2 + (\alpha - 1)^2 y^2)^2} &
       \frac{(\alpha - 1) \gamma (\kappa - 1)^2 \kappa x^2 y^2 (\alpha \kappa
       - 1)}{(\kappa^2 x^2 (\alpha \kappa - 1)^2 + (\alpha - 1)^2 y^2)^2}
     \end{array}\right)} . \]
In particular, for any $(s_1, s_2) \in \mathcal{S}$, the spectral radius of the product Jacobians satisfy
  \[ \rho (\mathcal{J}_{\mathcal{F}_1} (s_1) \mathcal{J}_{\mathcal{F}_1}
     (s_2)) = \frac{4( \kappa^2 - 4 \kappa + 1)^2} {(\kappa - 1)^4} > 1 \quad \text{and} \quad \rho (\mathcal{J}_{\mathcal{F}_2} (s_1) \mathcal{J}_{\mathcal{F}_2}
     (s_2)) = 1. \]
\end{proposition}

Proposition~\ref{prop:dynamical-system} shows that when $\gamma = 1$ and $\kappa > 6$, the spectral radius of the product Jacobian is strictly greater than $1$. Hence, the nonlinear dynamical system~\eqref{eqn:dynamic} is unstable around the trajectory induced by the worst-case function. Consequently, once floating-point errors enter the iterations, {\polyak} deviates from this trajectory and exhibits accelerated convergence. Similar observation holds for all $\gamma \in (0, 2)$. In contrast, when $\gamma = 2$, the spectral radius is exactly $1$, implying that $2$-\texttt{PolyakGD} may become trapped on the worst-case trajectory. It is consistent with the zig-zag behavior observed in \cite{huang2024analytic}. In summary, the worst-case trajectory is only realized under exact arithmetic, and \texttt{PolyakGD} in practice is largely immune to it.

%% file: sec_ub.tex
\section{Universal Convergence of {{\polyak}} Across Function Classes}\label{section:universality}
In this section, we demonstrate that {\polyak} is universal across various function classes, following the terminology ``universal" introduced in \cite{nesterov2015universal}. We first establish convergence under {\hd} smoothness and growth conditions, and then provide several extensions, such as relaxations of convexity and the stochastic Polyak stepsize.

\subsection{Convergence results under {\hd} smoothness and growth}

Before presenting the main convergence result, we justify the choice of the set $\cK$ for the {\hd} growth condition in \eqref{eq:Holder growth}. Given the parameter $\gamma \in (0, 2]$, first note that for any iteration $k \ge 0$ and any $x^\star \in \cX^\star$, the distance to $x^\star$ satisfies
    \begin{equation}\label{eq:bound optimality distance}
        \begin{aligned}
           \|x^{k+1} - x^\star\|^2 = \ &\|x^{k} - \gamma \cdot \alpha_k \nabla f(x^k)- x^\star\|^2 \\
        = \ &\|x^{k} - x^\star\|^2 - 2\gamma \cdot \alpha_k \langle \nabla f(x^k), x^{k} - x^\star \rangle + \gamma^2 \cdot \alpha_k^2 \|\nabla f(x^k)\|^2 \\
        \le \ &\|x^{k} - x^\star\|^2 - 2\gamma \cdot \alpha_k (f(x^k) - f^\star) + \gamma^2 \cdot \alpha_k^2 \|\nabla f(x^k)\|^2  \\
        = \ &\|x^{k} - x^\star\|^2 - \gamma(2-\gamma)\frac{(f(x^k) - f^\star)^2}{\|\nabla f(x^k)\|^2}. 
        \end{aligned}
    \end{equation}
This inequality implies that the iterate sequence generated by {\polyak} is Fej\'{e}r monotone. Accordingly, we define
\begin{equation}\label{eq:selection of cK}
    \cK(x^\star) \coloneqq\mathbb{B}(x^\star, \|x^0 - x^\star\|), \ \text{and} \ \cK \coloneqq\bigcup_{x^\star \in \cX^\star} \cK(x^\star),
\end{equation}
where $\mathbb{B}(x^\star, \|x^0 - x^\star\|)$ denotes the closed ball centered at $x^\star$ with radius $\|x^0 - x^\star\|$. 
Clearly, we have $\cX^\star \subseteq \bigcup_{x^\star \in \cX^\star} \cK(x^\star)$ and $\sup_{x \in \cK} \dist(x, \cX^\star) < \infty$, satisfying the requirement discussed in Section~\ref{section:preliminary}.
\begin{theorem}\label{thm:Holder smooth growth}
   Let $\gamma \in (0,2)$ and $x^0 \in \reals^n$. Suppose that the objective function $f$ satisfies the $(\nu, L_\nu)$-{\hd} smoothness condition (Definition \ref{def:smooth}) and the $(r, \rho_r)$-{\hd} growth condition (Definition \ref{def:growth}) on $\cK$ defined in \eqref{eq:selection of cK}. Then, for any (even) iteration count $K \ge 1$, the iterates generated by {\polyak} satisfy:
   
\begin{itemize}[leftmargin=12pt]
\item When $r = \nu + 1$,
\begin{align*}
       \min_{1 \le k \le K} \ f(x^k) - f^\star \le \Bigg(1 - \frac{\gamma(2-\gamma) \nu^{\frac{2\nu}{\nu + 1}}\rho_r^{\frac{2}{\nu + 1}}}{(\nu + 1)^{\frac{2\nu}{\nu + 1}}L_\nu^{\frac{2}{\nu + 1}}}\Bigg)^{\frac{K}{2}} \sup_{x \in {\cK}} \ \|\nabla f(x)\| \cdot \dist(x^0, \mathcal{X}^\star),
    \end{align*}
\item When $r > \nu + 1$,  
\begin{align*}
       \min_{1 \le k \le K} \ f(x^k) - f^\star \le \frac{(r - \nu - 1)^{\frac{(\nu + 1)^2}{2(\nu+1-r)}} (\gamma(2-\gamma))^{\frac{r(\nu + 1)}{2(\nu+1-r)}} \nu^{\frac{\nu r}{\nu+1-r}}}{2^{\frac{r(\nu + 1)}{2(\nu+1-r)}} (\nu + 1)^{\frac{(\nu+1)^2 + 2\nu r}{2(\nu+1-r)}}} \cdot \frac{L_\nu^{\frac{r}{r - \nu - 1}}}{\rho_r^{\frac{\nu + 1}{r - \nu - 1}} K^{\frac{r(\nu + 1)}{2(r - \nu - 1)}}}.
    \end{align*}
\end{itemize}
\end{theorem}
\begin{proof}
    Let $\D_k \coloneqq \dist(x^k, \cX^\star)$ and $\Delta_k \coloneqq f(x^k) - f^\star$ for short. From Equation \eqref{eq:bound optimality distance}, we obtain
    \begin{align*}
        \D_{k+1}^2 \le \|x^{k+1} - x^\star\|^2 \le \|x^k - x^\star\|^2 - \gamma(2-\gamma)\frac{\Delta_k^2}{\|\nabla f(x^k)\|^2}.
    \end{align*}
    Since the right-hand side in the above inequality holds for all $x^\star \in \cX^\star$, we derive
    \begin{equation}\label{eq:basic inequality}
        \D_{k+1}^2 \le \D_k^2 - \gamma(2-\gamma)\frac{\Delta_k^2}{\|\nabla f(x^k)\|^2}.
    \end{equation}
    Using the second inequality in Equation \eqref{eq:Holder smooth inequalities} yields
    \begin{align*}
       \D_{k+1}^2 \le \D_{k}^2 - \frac{\gamma(2-\gamma) \nu^{\frac{2\nu}{\nu + 1}}}{(\nu + 1)^{\frac{2\nu}{\nu + 1}}L_\nu^{\frac{2}{\nu + 1}}}\Delta_k^{\frac{2}{\nu + 1}}.  
    \end{align*}
    Now we substitute the {\hd} growth condition to get
    \begin{equation}\label{eq:Holder smooth growth inequality}
        \D_{k+1}^2 \le \D_{k}^2 - \frac{\gamma(2-\gamma) \nu^{\frac{2\nu}{\nu + 1}}\rho_r^{\frac{2}{\nu + 1}}}{(\nu + 1)^{\frac{2\nu}{\nu + 1}}L_\nu^{\frac{2}{\nu + 1}}}\D_k^{\frac{2r}{\nu + 1}}.
    \end{equation}
\paragraph{Case 1. $r = \nu + 1$.}  the above equation reduces to
    \begin{align*}
        \D_{k+1}^2 \le \D_{k}^2 - \frac{\gamma(2-\gamma) \nu^{\frac{2\nu}{\nu + 1}}\rho_r^{\frac{2}{\nu + 1}}}{(\nu + 1)^{\frac{2\nu}{\nu + 1}}L_\nu^{\frac{2}{\nu + 1}}}\D_k^{2}.
    \end{align*}
    Note that in this case, Equation \eqref{eq:relation of nu and rho} implies  $        \frac{\rho_r}{L_\nu} \le \frac{1}{\nu+1} \le 1$.
    Hence, the linear convergence follows
    \begin{align*}
        \D_{k+1}^2 \le \Bigg(1 - \frac{\gamma(2-\gamma) \nu^{\frac{2\nu}{\nu + 1}}\rho_r^{\frac{2}{\nu + 1}}}{(\nu + 1)^{\frac{2\nu}{\nu + 1}}L_\nu^{\frac{2}{\nu + 1}}}\Bigg)\D_k^{2},
    \end{align*}
    as these constants satisfy $\gamma(2-\gamma) < 1$, $\nu / (\nu + 1) < 1$ and $\rho_r / L_v \le 1$. By convexity, it holds that $\Delta_k \le \inner{\nabla f(x^k)}{x^k - x^\star} \le G \D_k$, $\text{ for all } k \ge 0$, where $G \coloneqq \sup_{x \in {\cK}} \ \|\nabla f(x)\|$. Therefore, we conclude
    \begin{align*}
        \min_{1 \le k \le K} \ f(x^k) - f^\star \le \Bigg(1 - \frac{\gamma(2-\gamma) \nu^{\frac{2\nu}{\nu + 1}}\rho_r^{\frac{2}{\nu + 1}}}{(\nu + 1)^{\frac{2\nu}{\nu + 1}}L_\nu^{\frac{2}{\nu + 1}}}\Bigg)^{\frac{K}{2}} G \D_0, \ \text{ for all } K \ge 1.
    \end{align*}
\paragraph{Case 2. $r > \nu + 1$.} From Equation \eqref{eq:Holder smooth growth inequality}, we use the technical lemma \ref{lemma:recursive technical lemma} in the Appendix (by setting $a_k = \D_k^2, \ \tau = \frac{r}{\nu + 1} > 1, \ c = \frac{\gamma(2-\gamma) \nu^{\frac{2\nu}{\nu + 1}}\rho_r^{\frac{2}{\nu + 1}}}{(\nu + 1)^{\frac{2\nu}{\nu + 1}}L_\nu^{\frac{2}{\nu + 1}}}$.)
    It immediately follows
    \begin{align*}
        \D_k^2 \le \ &\Bigg(\D_0^{\frac{2(\nu + 1 - r)}{\nu + 1}} + \frac{r - \nu - 1}{\nu + 1} \cdot \frac{\gamma(2-\gamma) \nu^{\frac{2\nu}{\nu + 1}}\rho_r^{\frac{2}{\nu + 1}}}{(\nu + 1)^{\frac{2\nu}{\nu + 1}}L_\nu^{\frac{2}{\nu + 1}}} k\Bigg)^{\frac{\nu + 1}{\nu + 1 - r}} \\
        \le \ &\Bigg(\frac{r - \nu - 1}{\nu + 1} \cdot \frac{\gamma(2-\gamma) \nu^{\frac{2\nu}{\nu + 1}}\rho_r^{\frac{2}{\nu + 1}}}{(\nu + 1)^{\frac{2\nu}{\nu + 1}}L_\nu^{\frac{2}{\nu + 1}}} k\Bigg)^{\frac{\nu + 1}{\nu + 1 - r}}.
    \end{align*}
    Now we combine Equation \eqref{eq:basic inequality} with the second inequality in Equation \eqref{eq:Holder smooth inequalities}, and it follows
    \begin{align*}
       \gamma(2-\gamma)\Delta_k^\frac{2}{\nu + 1} \le \frac{(\nu + 1)^{\frac{2\nu}{\nu + 1}}L_\nu^{\frac{2}{\nu + 1}}}{\nu^{\frac{2\nu}{\nu + 1}}}(\D_k^2 - \D_{k+1}^2).
    \end{align*}
    Without loss of generality, we choose $K$ as an even number such that $K \ge 2$. Then we get
    \begin{align*}
        \frac{1}{K/2}\sum_{k = K/2}^{K-1} \gamma(2-\gamma)\Delta_k^\frac{2}{\nu + 1} \le \ &\frac{2(\nu + 1)^{\frac{2\nu}{\nu + 1}}L_\nu^{\frac{2}{\nu + 1}}}{\nu^{\frac{2\nu}{\nu + 1}}K} \sum_{k = K/2}^{K-1} (\D_k^2 - \D_{k+1}^2) \\
        \le \ &\frac{2(\nu + 1)^{\frac{2\nu}{\nu + 1}}L_\nu^{\frac{2}{\nu + 1}}}{\nu^{\frac{2\nu}{\nu + 1}}K} \D_{K/2}^2 \\
        \le \ &\frac{2(\nu + 1)^{\frac{2\nu}{\nu + 1}}L_\nu^{\frac{2}{\nu + 1}}}{\nu^{\frac{2\nu}{\nu + 1}}K} \Bigg(\frac{r - \nu - 1}{\nu + 1} \cdot \frac{\gamma(2-\gamma) \nu^{\frac{2\nu}{\nu + 1}}\rho_r^{\frac{2}{\nu + 1}}}{2(\nu + 1)^{\frac{2\nu}{\nu + 1}}L_\nu^{\frac{2}{\nu + 1}}} K\Bigg)^{\frac{\nu + 1}{\nu + 1 - r}}.
    \end{align*}
    This implies
    \begin{align*}
         \min_{1 \le k \le K} \ (f(x^k) - f^\star)^{\frac{2}{\nu + 1}} \le \ &\frac{2(\nu + 1)^{\frac{2\nu}{\nu + 1}}L_\nu^{\frac{2}{\nu + 1}}}{\gamma(2-\gamma)\nu^{\frac{2\nu}{\nu + 1}}K} \Bigg(\frac{r - \nu - 1}{\nu + 1} \cdot \frac{\gamma(2-\gamma) \nu^{\frac{2\nu}{\nu + 1}}\rho_r^{\frac{2}{\nu + 1}}}{2(\nu + 1)^{\frac{2\nu}{\nu + 1}}L_\nu^{\frac{2}{\nu + 1}}} K\Bigg)^{\frac{\nu + 1}{\nu + 1 - r}} \\
        = \ &\frac{(r - \nu - 1)^{\frac{\nu+1}{\nu+1-r}} (\gamma(2-\gamma))^{\frac{r}{\nu+1-r}} \nu^{\frac{2\nu r}{(\nu+1)(\nu+1-r)}}}{2^{\frac{r}{\nu+1-r}} (\nu + 1)^{\frac{(\nu+1)^2 + 2\nu r}{(\nu+1)(\nu+1-r)}}} \cdot \frac{L_\nu^{\frac{2r}{(\nu+1)(r - \nu - 1)}}}{\rho_r^{\frac{2}{r - \nu - 1}} K^{\frac{r}{r - \nu - 1}}}.
    \end{align*}
    Therefore, we conclude
    \begin{align*}
       \min_{1 \le k \le K} \ f(x^k) - f^\star \le \frac{(r - \nu - 1)^{\frac{(\nu + 1)^2}{2(\nu+1-r)}} (\gamma(2-\gamma))^{\frac{r(\nu + 1)}{2(\nu+1-r)}} \nu^{\frac{\nu r}{\nu+1-r}}}{2^{\frac{r(\nu + 1)}{2(\nu+1-r)}} (\nu + 1)^{\frac{(\nu+1)^2 + 2\nu r}{2(\nu+1-r)}}} \cdot \frac{L_\nu^{\frac{r}{r - \nu - 1}}}{\rho_r^{\frac{\nu + 1}{r - \nu - 1}} K^{\frac{r(\nu + 1)}{2(r - \nu - 1)}}}.
    \end{align*}
    The proof is completed.
\end{proof}

\begin{remark}
    The above theorem demonstrates that {\polyak} adapts to both {\hd} smoothness and {\hd} growth conditions. Specifically, letting $r \rightarrow \infty$, the above result reduces to $\cO(1/K^{(\nu + 1) / 2})$\footnote{\cite{orabona2025new} established a more general result for the stochastic setting under {\hd} smoothness with arbitrary surrogate functions; see Theorem 7 therein.}, which matches the rate achieved by the universal gradient method in \cite{nesterov2015universal}. This result is shown to be tight in Theorem \ref{thm:Holder smooth-cvx-lb}. On the other hand, setting $\nu = 0$, the above result reduces to $\cO(1/K^{r / (2(r - 1))})$, which matches the lower bound \citep{nemirovskii1985optimal}. This rate is also achieved by restart-based optimization methods \citep{roulet2017sharpness,yang2018rsg,xu2019accelerating,d2021acceleration,renegar2022simple,grimmer2023general}.
\end{remark}

\subsection{Extensions}
In this subsection, we extend the previous analysis to more general settings.

\paragraph{Star-convexity.} The proof of Theorem \ref{thm:Holder smooth growth} relies on the convexity of $f$ to obtain\[\inner{\nabla f(x^k)}{x^k - x^\star} \ge f(x^k) - f^\star\] in Equation \eqref{eq:bound optimality distance}. This inequality corresponds to a weaker condition known as star-convexity (see Definition 4.1 in \cite{nesterov2018lectures}, for example). Therefore, the convexity assumption can naturally be relaxed to star-convexity. For example, when $f(x) \coloneqq h(g(x))$ is the composition between an outer scalar function $h: \reals \rightarrow \reals$ and a function $g: \reals^n \rightarrow \reals$, $f$ can be shown to be star-convex under mild conditions of $h$ and $g$.  
\begin{lemma}\label{lemma:star convexity of f}
    Suppose the scalar function $h$ is nondecreasing and convex, and the function $g$ is differentiable and (star-)convex, i.e., $\inner{\nabla g(x)}{x - x_g^\star} \ge g(x) - g(x_g^\star)$, where $x_g^\star \in \cX_g^\star \coloneqq\argmin_{x \in \reals^n} \ g(x)$. For any $x_g^\star \in \cX_g^\star$, the objective function $f$ satisfies $\inner{\xi(x)\nabla g(x)}{x - x_g^\star} \ge f(x) - f(x_g^\star)$, for all $x \in \reals^n$, and $\xi(x) \in \partial h(g(x))$.
\end{lemma}

This structure is common in machine learning optimization tasks \citep{chen2021solving,jiang2022optimal}, where the function $f$ itself may be nonconvex and nonsmooth. Star-convexity suffices to ensure that Theorem \ref{thm:Holder smooth growth} remains valid. The only modification is that we need to use the subgradient of $f$ in {\polyak} (Algorithm \ref{alg:polyakgd}).

\paragraph{Convergence of $2$-\texttt{PolyakGD}.} Note that Theorem \ref{thm:Holder smooth growth} is not applicable to the case $\gamma = 2$. To establish a convergence guarantee for this case, we introduce the following inequality. This inequality is an extension of inequality (2.1.10)\footnote{This inequality is also known as cocoercivity in the literature (see, e.g., \cite{malitsky2024adaptive,latafat2024adaptive}).} in \cite{nesterov2018lectures} to {\hd} smoothness. 
\begin{lemma}\label{lemma:cocoercivity Holder inequality}
  Suppose $f : \mathbb{R}^n \rightarrow \mathbb{R}$ is $(L_{\nu}, \nu)$-{\hd}
  smooth with $\nu \in (0,1]$ and $L_\nu > 0$, then it holds that
  \[ f (y) \geq f (x) + \langle \nabla f (x), y - x \rangle +
     \frac{\nu}{\nu + 1} \frac{1}{L_{\nu}^{\frac{1}{\nu}}} \| \nabla f (x) - \nabla f
     (y) \|^{\frac{\nu + 1}{\nu}}, \ \text{ for all } x, y \in \reals^n.\]
\end{lemma}

\begin{theorem}\label{thm:two times polyak}
      Suppose the objective function $f$ satisfies $(\nu, L_\nu)$-{\hd} smoothness. For any total iteration count $K \ge 1$, the iterates generated by {\polyak} with $\gamma = 2$ satisfy
    \begin{align*}
       \min_{1 \le k \le K} \ f(x^k) - f^\star \le \Big(\frac{\nu + 1}{4\nu}\Big)^\nu \frac{L_\nu \dist(x^0, \cX^\star)^{\nu+1}}{K^\nu}.
    \end{align*}
\end{theorem}

\begin{remark}
When $\gamma=2$, the upper bound $\cO(1/K^\nu)$ is slightly weaker than the result $\cO(1/K^{(\nu + 1)/2})$ in Theorem \ref{thm:Holder smooth growth} for $\nu \in (0,1]$. In the case $\nu = 1$ (i.e., $L$-smoothness), the two bounds coincide. We argue that the rate $\cO(1/K^\nu)$ is reasonable because when $\nu \rightarrow 0$ (the nonsmooth case), $2$-\texttt{PolyakGD} may not converge. For example, consider the absolute value function $|x|$ with initial point $x^0 = 1$; $2$-\texttt{PolyakGD} oscillates between the two points $-1$ and $1$.
\end{remark}

As a byproduct, the tighter inequality in Lemma \ref{lemma:cocoercivity Holder inequality} can also be used to establish an $\cO(1/K^{\nu/2})$ convergence rate for the gradient norm. According to Theorem \ref{thm:smooth-cvx-lb gradient norm}, this upper bound is tight.
\begin{theorem}\label{thm:gradient norm convergence}
    Suppose the objective function $f$ satisfies $(\nu, L_\nu)$-{\hd} smoothness. For any total iteration count $K \ge 1$, the iterates generated by {\polyak} with $\gamma \in (0, 2]$ satisfy
    \begin{align*}
        \min_{1 \le k \le K} \ \|\nabla f(x^k)\| \le \frac{(\nu+1)^{\nu} L_\nu \dist(x^0, \cX^\star)^{\nu}}
         {2^{\frac{\nu}{2}} \gamma^{\frac{\nu}{2}} \nu^{\nu} K^{\frac{\nu}{2}}}.
    \end{align*}
\end{theorem}

\paragraph{Global curvature bound.} We now consider a more general upper curvature bound than {\hd} smoothness, known as the global curvature bound \citep{nesterov2025universal,doikov2025universal}. This bound is intrinsic as it does not require the objective function $f$ to belong to any specific function class; it is defined by the following quantity:
\begin{equation}\label{eq:gcb}
    \hat{\mu}_f(t) \coloneqq \sup_{ \substack{ \|x-y\| \le t, \\ \alpha \in (0,1)} } \frac{|\alpha f(x) + (1-\alpha)f(y) - f(\alpha x + (1-\alpha)y)|}{\alpha(1-\alpha)}.
\end{equation}
Given $\hat{\mu}_f$, its inverse function  $s_f$, known as the complexity gauge, corresponds to a stationary criterion for $f$ (see Definition 1 in \cite{nesterov2025universal}). Specifically, it satisfies $\hat{\mu}_f(s_f(t)) = t$ for all $t \ge 0$. The following theorem shows that {\polyak} also adapts to this global curvature bound, matching the result of the universal primal gradient method in \cite{nesterov2025universal}.
\begin{theorem}\label{thm:gcb convergence}
    Let $\gamma \in (0,2)$ and $x^0 \in \reals^n$. Suppose $\hat{\mu}_f $ is well-defined for any $t \in [0, \infty)$ and satisfies $\lim_{t \rightarrow 0} \hat{\mu}(t) / t = 0$. For any total iteration count $K \ge 1$, the iterates generated by {\polyak} satisfy
    \begin{align*}
       \min_{1 \le k \le K} \ f(x^k) - f^\star \le \hat{\mu}_f\left(\frac{3  \dist(x^0, \cX^\star)^2}{\sqrt{K}}\right).
    \end{align*}
    Consequently, given an accuracy $\varepsilon > 0$, {\polyak} guarantees $\min_{1 \leq k \leq K} \ f(x^k) - f^\star \le \varepsilon$ after at most $K = \cO(\dist(x^0, \cX^\star)^2/s^2_f(\varepsilon))$ iterations.
\end{theorem}

\paragraph{Stochastic Polyak stepsize.} To conclude this section, we extend {\polyak} to stochastic Polyak stepsize \citep{loizou2021stochastic}. Consider the following stochastic optimization problem
\[ \min_{x \in \mathbb{R}^{n}} \ f(x) \assign \mathbb{E}_\xi [f (x, \xi)],\]
where $f $ is assumed to satisfy the interpolation condition: $\min_x \mathbb{E} [f (x, \xi)] -\mathbb{E}
[\min_x f (x, \xi)] = 0$. The stochastic {\polyak} update is given by
\begin{align*}
    x^{k+1} = x^k - \gamma \cdot \frac{f (x^k, \xi^k) - \min_x f(x, \xi^k)}{\|\nabla f (x^k, \xi^k)\|^2} \nabla f (x^k, \xi^k),
\end{align*}
where $\mathbb{E}[\nabla f (x^k, \xi^k) | x^k] = \nabla f(x^k)$.
The convergence guarantees of stochastic {\polyak} resemble those of the deterministic case.
\begin{theorem} \label{thm:stochastic-polyak} Suppose the interpolation condition holds and that $\mathcal{X}^{\star} = \{
x^{\star} \}$ is a singleton. Then stochastic {\polyak} has the same rates as Theorem \ref{thm:Holder smooth growth} in terms of $\mathbb{E} [f(x^k) - f^\star]$ .
\end{theorem}

%% file: sec_appendix.tex
\section*{Appendix}

\section{Technical Lemmas}
\begin{lemma}\label{lemma:lb e-gamma} For any $\gamma \in (0, 4)$, the following relation holds for all $K \geq 1$:

\[\left(\frac{8K-2(1 + K)\gamma + \gamma^2}{2K(4-\gamma) + 2\gamma}\right)^{2K} \geq \bigg(\frac{K}{\gamma + K}\bigg)^{2K} \geq \lim_{K \rightarrow \infty} \bigg(\frac{K}{\gamma + K}\bigg)^{2K} = \ee^{-2\gamma}.\]

\end{lemma}

\begin{proof}
The proof is available at
	\href{https://github.com/Gwzwpxz/polyak_proof}{this link}.
\end{proof}

\begin{lemma}\label{lemma:recursive technical lemma}
    Let $\{a_k\}_{k \ge 0}$ be a nonnegative sequence satisfying $a_{k+1} \le a_k - c \cdot a_k^\tau$ for some $c > 0$ and $\tau > 1$. Then it holds that $a_k \le \big(a_0^{1-\tau} + (\tau-1)c k\big)^{-\frac{1}{\tau-1}}$.
\end{lemma}

\begin{proof}
    Let $\phi(t) \coloneqq t^{1-\tau}$, $t > 0$. Note that $\phi $ is strictly decreasing over $(0,\infty)$ when $\tau > 1$ since $\phi^\prime(t) = (1-\tau)t^{-\tau } < 0$. Fix $k \ge 0$. Recall the sequence satisfies $a_k - a_{k+1} \ge c \cdot a_k^\tau  > 0$. Then by Lagrange's mean value theorem, there exists a value $\xi_k \in [a_{k+1}, a_k]$ such that
    \begin{align*}
        a_{k+1}^{1-\tau} - a_k^{1-\tau} = \phi(a_{k+1}) - \phi(a_k) = \phi^\prime(\xi_k)(a_{k+1} - a_k) = (\tau-1)\xi_k^{-\tau}(a_k - a_{k+1}).
    \end{align*}
    Since $\phi $ is decreasing, then $\phi(\xi_k) \ge \phi(a_k)$, which yields that
    \begin{align*}
        a_{k+1}^{1-\tau} - a_k^{1-\tau} \ge (\tau-1)a_k^{-\tau}(a_k - a_{k+1}) \ge (\tau-1)a_k^{-\tau} c \cdot a_k^\tau = (\tau-1)c.
    \end{align*}
    Therefore, we directly obtain $a_k^{1-\tau} \ge a_0^{1-\tau} + (\tau-1)ck$, which implies that
    \begin{align*}
       a_k \le \big(a_0^{1-\tau} + (\tau-1)c k\big)^{-\frac{1}{\tau-1}}.
    \end{align*}
    The proof is completed.
\end{proof}

\begin{lemma}\label{lemma:Holder smooth |Ax|}
    For any symmetric matrix $A \in \reals^{n \times n}$ and $\nu \in (0,1]$, the function $p_A(x) = \|A x\|^{1+\nu}$ is $(\nu, L_\nu)$-{\hd} smooth, where $L_\nu = 2^{1-\nu}(\nu + 1)\|A\|^{\nu + 1}$.
\end{lemma}
\begin{proof}
    According to Theorem 6.3 in \cite{rodomanov2020smoothness}, the function $p(x) = \|x\|^{1+\nu}$ is $(\nu,2^{1-\nu}(\nu + 1))$-{\hd} smooth. Note that $\nabla p_A(x) = \nabla (p(Ax)) = A \nabla p(Ax)$. Hence, for any $x, y \in \reals^n$, we have
    \begin{align*}
        \|\nabla p_A(x) - \nabla p_A(y)\| = \ &\|A(\nabla p(Ax) - \nabla p(Ay))\| \\
        \le \ &\|A\| \cdot 2^{1-\nu}(\nu + 1)\|Ax - Ay\|^\nu \\
        \le \ &2^{1-\nu}(\nu + 1)\|A\|^{\nu + 1}\|x - y\|^\nu.
    \end{align*}
    Setting $L_\nu = 2^{1-\nu}(\nu + 1)\|A\|^{\nu + 1}$ completes the proof.
\end{proof}

\section{Missing Proofs in Section \ref{section:tight convergence}}
\subsection{Proof of Proposition \ref{prop:Huber loss}}

   First, observe that the Huber loss is a quadratic function with condition number $1$ when $\|x\| \le 1/(2K+1)$. By Theorem 1 in \cite{hazan2019revisiting}, 1-\texttt{PolyakGD} converges to $x^\star$ in one step once the iterate enters this region. Hence, it suffices to consider an initial point $x^0$ satisfying $\|x^0\| > 1/(2K+1)$. We now show that after a single iteration, the iterate $x^1$ generated by $1$-\texttt{PolyakGD} satisfies $\|x^1\| \le 1/(2K+1)$. Specifically, the stepsize is given by
   \begin{align*}
       \alpha_0 = \frac{\cH_K(x^0) - \cH_K^\star}{\|\nabla \cH_K(x^0)\|^2} = \frac{\frac{1}{2K+1}\|x^0\| - \frac{1}{2(2K+1)^2}}{\frac{1}{(2K+1)^2}} = (2K+1)\|x^0\| - \frac{1}{2}.
   \end{align*}
   Substituting the stepsize into the update rule yields
   \begin{align*}
       x^1 = x^0 - \alpha_0  \nabla \cH_K(x^0) = x^0 - (2K+1)\|x^0\| \cdot \frac{x^0}{(2K+1)\|x^0\|} + \frac{1}{2}\frac{x^0}{(2K+1)\|x^0\|} = \frac{x^0}{2(2K+1)\|x^0\|}.
   \end{align*}
   Clearly, we have $\|x^1\| = 1/(2(2K+1)) < 1/(2K+1)$, which completes the proof.

\subsection{Proof of Theorem \ref{thm:Holder smooth-cvx-lb}}

    Let $\Tilde{q} $ be the quadratic function defined in  Equation \eqref{eq:scaled lower bound function} and Theorem \ref{thm:smooth-cvx-lb}. Define $q_{\nu}(x) \coloneqq\Tilde{q}(x)^{\frac{\nu+1}{2}}$. Then, we have
   \begin{align*}
      \nabla q_\nu(x) = \frac{\nu+1}{2}\Tilde{q}(x)^{\frac{\nu-1}{2}}\nabla \Tilde{q}(x).
   \end{align*}
    Consequently, the $\gamma$-\texttt{PolyakGD} update for $q_{\nu} $ is given by:
    \begin{align*}
        x^{k+1} = \ &x^k - \gamma \cdot \frac{q_\nu(x^k)}{\|\nabla q_\nu(x^k)\|^2} \nabla q_\nu(x^k) \\
        = \ &x^k - \gamma \cdot \frac{\Tilde{q}(x^k)^{\frac{\nu+1}{2}}}{\left(\frac{\nu + 1}{2}\right)^2 \Tilde{q}(x^k)^{\nu-1}} \frac{\nu+1}{2}\Tilde{q}(x^k)^{\frac{\nu-1}{2}}\frac{1}{\|\nabla \Tilde{q}(x^k)\|^2}\nabla \Tilde{q}(x^k) \\
        = \ &x^k - \frac{2\gamma}{\nu + 1} \frac{\Tilde{q}(x^k)}{\|\nabla \Tilde{q}(x^k)\|^2}\nabla \Tilde{q}(x^k).
    \end{align*}
    Observe that the trajectory coincides with that of $\frac{2\gamma}{\nu + 1}$-\texttt{PolyakGD} applied to $\Tilde{q} $. Given that $\frac{2\gamma}{\nu + 1} \in (0,4)$ for all $\gamma \in (0,2)$ and $\nu \in (0,1]$, a direct application of Theorem \ref{thm:smooth-cvx-lb} establishes shows, for a specific initial point $x^0$, that
    \begin{align*}
         \min_{1 \le k \le K} \ \frac{\Tilde{q}(x^{k}) - \Tilde{q}^\star}{L \|x^0 - x^{\star}\|^2} \geq \frac{\gamma}{2\ee^{\frac{4\gamma}{\nu + 1}}((2(\nu + 1) - \gamma)K + \gamma)}.
    \end{align*}
Using the fact that $q_\nu^\star = \Tilde{q}^\star = 0$,  we obtain
    \begin{equation}\label{eq:lb Holder smooth}
        \min_{1 \le k \le K} \ q_{\nu}(x^k) = \min_{1 \le k \le K} \ \Tilde{q}(x^k)^{\frac{\nu+1}{2}} \ge \Bigg(\frac{\gamma L\|x^0 - x^\star\|^2}{2\ee^{\frac{4\gamma}{\nu + 1}}((2(\nu + 1) - \gamma)K + \gamma)}\Bigg)^{\frac{\nu+1}{2}}.
    \end{equation}
    Recall that
    \begin{align*}
        \Tilde{q}(x) = \frac{L}{\kappa} q(x) = \frac{1}{2}\langle x, A x \rangle = \frac{1}{2}\|A^{1/2} x\|^2, \ A = \begin{pmatrix}
    L & 0 \\
    0 & L/\kappa
    \end{pmatrix}.
    \end{align*}
    Hence, we have $q_{\nu}(x) = 2^{-\frac{\nu + 1}{2}}\|A^{1/2}x\|^{\nu + 1}$. By technical lemma \ref{lemma:Holder smooth |Ax|}, $q_\nu $ is $(\nu,L_\nu)$-{\hd} smooth with \[L_\nu = 2^{-\frac{\nu + 1}{2}} 2^{1-\nu}(\nu + 1)\|A^{1/2}\|^{\nu + 1} = 2^{\frac{-3\nu + 1}{2}}(\nu + 1)L^{\frac{\nu + 1}{2}}.\] Finally, combining this with Equation \eqref{eq:lb Holder smooth} yields
    \begin{align*}
       \min_{1 \le k \le K} \ \frac{q_\nu(x^k) - q_\nu^\star}{L_\nu\|x^0 - x^\star\|^{\nu + 1}} \ge \frac{2^{\nu - 1}\gamma^{\frac{\nu + 1}{2}}}{\ee^{2\gamma} (\nu + 1)((2(\nu + 1) - \gamma)K + \gamma)^{\frac{\nu + 1}{2}}}.
    \end{align*} 
    This completes the proof.

\subsection{Proof of Theorem \ref{thm:smooth-cvx-lb gradient norm}}

    Let us start with the $L$-smooth convex case. Consider the same setting as in Theorem \ref{thm:smooth-cvx-lb}. By definition of $\Tilde{q} $,
    \begin{align*}
        \|\nabla \Tilde{q}(x^k)\|^2 = \frac{L^2}{\kappa^2}\|\nabla q(x^k)\|^2 = \frac{L^2}{\kappa^2}\left(\kappa^2 (x_{(1)}^k)^2 + (x_{(2)}^k)^2\right), \ \text{ for all } k \ge 0.
    \end{align*}
    Recall that the trajectory satisfies the relationship
    \begin{align*}
       (x_{(2)}^k)^2 = \frac{4 \kappa^2 - \gamma \kappa (\kappa + 1)}{\gamma ( \kappa + 1) - 4}  (x_{(1)}^k)^2, \ \left| \frac{x_{(1)}^{k + 1}}{x_{(1)}^k} \right| = \frac{\kappa - 1}{\kappa + 1}, \ \text{ for all } k \ge 0. 
    \end{align*}
    Consequently, for all $k \ge 1$, the trajectory generated by {\polyak} satisfies
    \begin{equation}\label{eq:gradient norm k tilde q}
       \|\nabla \Tilde{q}(x^k)\|^2 = \Big(\frac{\kappa - 1}{\kappa + 1}\Big)^{2}\|\nabla \Tilde{q}(x^{k-1})\|^2 =\Big(\frac{\kappa - 1}{\kappa + 1}\Big)^{2 k}\|\nabla \Tilde{q}(x^{0})\|^2 = \Big(\frac{\kappa - 1}{\kappa + 1}\Big)^{2 k}\frac{\gamma L^2}{(4 - \gamma) \kappa}. 
    \end{equation}
    According to the proof of Theorem \ref{thm:smooth-cvx-lb}, for all $k \ge 1$, the function value satisfies
    \begin{equation}\label{eq:function value k tilde q}
       \Tilde{q}(x^{k}) = \Big(\frac{\kappa - 1}{\kappa + 1}\Big)^{2 k} \Tilde{q}(x^0) = \Big(\frac{\kappa - 1}{\kappa + 1}\Big)^{2 k}\frac{2L}{(4-\gamma)(\kappa + 1)}. 
    \end{equation}
    Following a similar argument to that in the proof of Theorem \ref{thm:Holder smooth-cvx-lb}, let $q_\nu(x) = \Tilde{q}(x)^{\frac{\nu + 1}{2}}$. Recall from the proof of Theorem \ref{thm:Holder smooth-cvx-lb} that the trajectory of {\polyak} generated on $q_\nu$ coincides with that of $\frac{2\gamma}{\nu + 1}$-\texttt{PolyakGD} applied to $\Tilde{q}$. Since $\nabla q_\nu(x) = \frac{\nu+1}{2}\Tilde{q}(x)^{\frac{\nu-1}{2}}\nabla \Tilde{q}(x)$, by substituting $\frac{2\gamma}{\nu + 1}$ for $\gamma$ in Equations \eqref{eq:gradient norm k tilde q} and \eqref{eq:function value k tilde q}, we obtain that the trajectory satisfies
     \begin{align*}
         \|\nabla q_\nu(x^k)\| = \ &\frac{\sqrt{\gamma}(\nu+1)^{\frac{\nu+1}{2}} L^{\frac{\nu+1}{2}}}
{2\sqrt{\kappa}\,\big(2(\nu+1)-\gamma\big)^{\frac{\nu}{2}} (\kappa+1)^{\frac{\nu-1}{2}}}
\Big(\frac{\kappa-1}{\kappa+1}\Big)^{k\nu} \\
        = \ &\frac{2^{\frac{3(\nu-1)}{2}}\sqrt{\gamma}\,(\nu+1)^{\frac{\nu-1}{2}} L_\nu\|x^0 - x^\star\|^{\frac{\nu + 1}{2}}}
{\sqrt{\kappa}\,\big(2(\nu+1)-\gamma\big)^{\frac{\nu}{2}}(\kappa+1)^{\frac{\nu-1}{2}}}
\Big(\frac{\kappa-1}{\kappa+1}\Big)^{k\nu}, \ \text{ for all } k \ge 1,
     \end{align*}
where the second equality holds since $L_\nu = 2^{\frac{-3\nu + 1}{2}}(\nu + 1)L^{\frac{\nu + 1}{2}}$, $\|x^0\| = 1$ and $\|x^\star\| = 0$. Note that the gradient norm is monotone in $k$. Choosing $\kappa = 4 K / \gamma + \gamma/(4-\gamma)$, we conclude
    \begin{align*}
        &\min_{1 \le k \le K} \ \frac{\|\nabla q_\nu(x^k)\|}{L_\nu \|x^0 - x^\star\|^{\frac{\nu + 1}{2}}} \\
        = \ &\frac{2^{\frac{\nu-1}{2}} \gamma^{\frac{\nu+1}{2}} (4-\gamma)^{\frac{\nu}{2}}
(\nu+1)^{\frac{\nu-1}{2}}}
{\sqrt{4K(4-\gamma)+\gamma^2} 
\big((4-\gamma)K+\gamma\big)^{\frac{\nu-1}{2}}
\big(2(\nu+1)-\gamma\big)^{\frac{\nu}{2}}}
\left(
\frac{4K - K\gamma - \gamma + \frac{\gamma^2}{2}}
{(4-\gamma)K+\gamma}
\right)^{K\nu} \\
     \ge \ &\frac{2^{\frac{\nu-1}{2}}\gamma^{\frac{\nu+1}{2}}
(\nu+1)^{\frac{\nu-1}{2}}\sqrt{4-\gamma}}
{\ee^{\gamma\nu}\sqrt{4(4-\gamma)+\gamma^2}(2(\nu+1)-\gamma)^{\frac{\nu}{2}}K^{\frac{\nu}{2}}}.
    \end{align*}
    The proof is completed.

\section{Missing Proofs in Section \ref{section:universality}}
\subsection{Proof of Lemma \ref{lemma:star convexity of f}}

    First note that 
    \begin{align*}
        h(g(x_g^\star)) \ge h(g(x)) + \xi(x)(g(x_g^\star) - g(x)), \ \text{ for all } \xi(x) \in \partial h(g(x)),
    \end{align*}
    by the convexity of $h$. Also, since the function $h$ is nondecreasing, we have $\xi(x) \ge 0$, which implies
    \begin{align*}
        f(x) - f(x^\star_g) \le \xi(x)(g(x) - g(x_g^\star)) \le \xi(x)(\inner{\nabla g(x)}{x - x_g^\star}).
    \end{align*}
    This completes the proof.

\subsection{Proof of Lemma \ref{lemma:cocoercivity Holder inequality}}

    Define $x^+ \coloneqq\argmin_x  \{f (y) + \langle \nabla f (y), x - y
\rangle + \frac{L_{\nu}}{\nu + 1} \| x - y \|^{\nu + 1} \}$. By the optimality condition,
\[ x^+ = y - \frac{1}{L_{\nu}^{\frac{1}{\nu}}} \| \nabla f (y) \|^{\frac{1 -
   \nu}{\nu}} \nabla f (y) . \]
Hence, it follows that
\begin{align}
  f (x^+) \leq \ &f (y) + \langle \nabla f (y), - \frac{1}{L_{\nu}^{\frac{1}{\nu}}}
  \| \nabla f (y) \|^{\frac{1 - \nu}{\nu}} \nabla f (y) \rangle +
  \frac{L_{\nu}}{\nu + 1} \| \frac{1}{L_{\nu}^{\frac{1}{\nu}}} \| \nabla f (y)
  \|^{\frac{1 - \nu}{\nu}} \nabla f (y) \|^{\nu + 1} \nonumber\\
  = \ &f (y) - \frac{\nu}{\nu + 1} \frac{1}{L_{\nu}^{\frac{1}{\nu}}} \| \nabla f (y)
  \|^{\frac{\nu + 1}{\nu}} . \nonumber
\end{align}
Next, define $\varphi (y) = f (y) - \langle \nabla f (x), y \rangle$. We have $x = \argmin_y \varphi (y)$ and
\begin{align}
  f (x) - \langle \nabla f (x), x \rangle = \ &\varphi (x) \nonumber\\
  \leq \ &\varphi \big( y - \frac{1}{L_{\nu}^{\frac{1}{\nu}}} \| \nabla f (y)
  \|^{\frac{1 - \nu}{\nu}} \big) \nonumber\\
  \leq \ &f (y) - \langle \nabla f (x), y \rangle - \frac{\nu}{\nu + 1}
  \frac{1}{L_{\nu}^{\frac{1}{\nu}}} \| \nabla f (x) - \nabla f (y) \|^{\frac{1 +
  \nu}{\nu}} . \nonumber
\end{align}
Rearranging the terms gives
\[ f (y) \geq f (x) + \langle \nabla f (x), y - x \rangle + \frac{\nu}{\nu
   + 1} \frac{1}{L_{\nu}^{\frac{1}{\nu}}} \| \nabla f (x) - \nabla f (y) \|^{\frac{1 +
   \nu}{\nu}} \]
and this completes the proof.

\subsection{Proof of Theorem \ref{thm:two times polyak}}
Using Lemma \ref{lemma:cocoercivity Holder inequality}, we deduce that
    \begin{align*}
        \|x^{k+1} - x^\star\|^2 = \ &\|x^{k} - 2\alpha_k \nabla f(x^k)- x^\star\|^2 \\
        = \ &\|x^{k} - x^\star\|^2 - 4 \alpha_k \langle \nabla f(x^k), x^{k} - x^\star \rangle + 4 \alpha_k^2 \|\nabla f(x^k)\|^2 \\
        \le \ &\|x^{k} - x^\star\|^2 - 4 \alpha_k \Big(f(x^k) - f^\star +
         \frac{\nu}{\nu + 1} \frac{1}{L_{\nu}^{\frac{1}{\nu}}} \| \nabla f (x^k)\|^{\frac{\nu + 1}{\nu}}\Big) + 4 \alpha_k^2 \|\nabla f(x^k)\|^2  \\
        = \ &\|x^{k} - x^\star\|^2 - \frac{4\nu}{\nu + 1} \frac{1}{L_{\nu}^{\frac{1}{\nu}}} (f(x^k) - f^\star)\|\nabla f(x^k)\|^{\frac{1 - \nu}{\nu}}.
    \end{align*}
    By a similar argument in the proof of Theorem \ref{thm:Holder smooth growth}, we have
    \begin{align*}
        \D_{k+1}^2 \le \D_k^2 - \frac{4\nu}{\nu + 1} \frac{\Delta_k}{L_{\nu}^{\frac{1}{\nu}}} \|\nabla f(x^k)\|^{\frac{1 - \nu}{\nu}},
    \end{align*}
    where $\D_k = \dist(x^k, \cX^\star)$ and $\Delta_k = f(x^k) - f^\star$. Note that due to convexity and Fej\'{e}r monotonicity, we have $\Delta_k \le \|\nabla f(x^k)\| \|x^k - x^\star\| \le \|\nabla f(x^k)\| \|x^0 - x^\star\|$ for any $x^\star \in \cX^\star$. Since $\nu \in (0,1]$, substituting this inequality back yields
    \begin{align*}
        \D_{k+1}^2 \le \D_{k}^2 - \frac{4\nu}{\nu + 1} \frac{\D_0^{\frac{\nu - 1}{\nu}}}{L_{\nu}^{\frac{1}{\nu}}} \Delta_k^\frac{1}{\nu}.
    \end{align*}
    By summining over $k$ from $k = 0$ to $k = K - 1$, we obtain
    \begin{align*}
        \sum_{k=0}^{K-1} \Delta_k^\frac{1}{\nu} \le \frac{(\nu + 1)L_{\nu}^{\frac{1}{\nu}}\D_0^{\frac{1 - \nu}{\nu}}}{4\nu}D_0^2,
    \end{align*}
    which immediately leads to
    \begin{align*}
        \min_{1 \le k \le K} \ f(x^k) - f^\star \le \Big(\frac{\nu + 1}{4\nu}\Big)^\nu \frac{L_\nu \D_0^{\nu+1}}{K^\nu}.
    \end{align*}
    This completes the proof.

\subsection{Proof of Theorem \ref{thm:gradient norm convergence}}
Using Lemma \ref{lemma:cocoercivity Holder inequality}, we deduce 
\begin{align*}
    \|x^{k+1} - x^\star\|^2 = \ &\|x^{k} - \gamma \cdot \alpha_k \nabla f(x^k)- x^\star\|^2 \\
        = \ &\|x^{k} - x^\star\|^2 - 2\gamma \cdot \alpha_k \langle \nabla f(x^k), x^{k} - x^\star \rangle + \gamma^2 \cdot \alpha_k^2 \|\nabla f(x^k)\|^2 \\
        \le \ &\|x^{k} - x^\star\|^2 - 2\gamma \cdot \alpha_k\Big(f(x^k) - f^\star +
         \frac{\nu}{\nu + 1} \frac{1}{L_{\nu}^{\frac{1}{\nu}}} \| \nabla f (x^k)\|^{\frac{\nu + 1}{\nu}}\Big) + \gamma^2 \cdot \alpha_k^2 \|\nabla f(x^k)\|^2  \\
        = \ &\|x^{k} - x^\star\|^2 - \gamma(2-\gamma)\frac{(f(x^k) - f^\star)^2}{\|\nabla f(x^k)\|^2} - 2\gamma \frac{f(x^k) - f^\star}{\|\nabla f(x^k)\|^2}  \frac{\nu}{\nu + 1} \frac{1}{L_{\nu}^{\frac{1}{\nu}}} \| \nabla f (x^k)\|^{\frac{\nu + 1}{\nu}} \\
        \le \ &\|x^{k} - x^\star\|^2 - \frac{2\gamma\nu}{\nu + 1} \frac{1}{L_{\nu}^{\frac{1}{\nu}}}(f(x^k) - f^\star)\| \nabla f (x^k)\|^{\frac{1 - \nu}{\nu}}
\end{align*}
Then by substituting the inequality in Lemma \ref{lemma:cocoercivity Holder inequality} with $x = x^\star$ and $y = x^k$, it holds that
\begin{align*}
    \Delta_k = f(x^k) - f^\star \ge \frac{\nu}{\nu + 1} \frac{1}{L_{\nu}^{\frac{1}{\nu}}} \| \nabla f (x^k)\|^{\frac{1 + \nu}{\nu}}.
\end{align*}
Consequently, we have
\begin{align*}
    \frac{2\gamma\nu^2}{(\nu + 1)^2} \frac{1}{L_{\nu}^{\frac{2}{\nu}}} \| \nabla f (x^k)\|^{\frac{2}{\nu}} \le \|x^k - x^\star\|^2 - \|x^{k+1} - x^\star\|^2 \le \|x^k - x^\star\|^2 - \D_{k+1}^2.
\end{align*}
Since $x^\star$ is taken arbitrarily, it implies
\begin{align*}
    \frac{2\gamma\nu^2}{(\nu + 1)^2} \frac{1}{L_{\nu}^{\frac{2}{\nu}}} \| \nabla f (x^k)\|^{\frac{2}{\nu}} \le \D_k^2 - \D_{k+1}^2.
\end{align*}
Summing over $k$ from $k = 0$ to $k = K-1$ and dividing $K$ on both sides, we conclude
\begin{align*}
    \min_{1 \le k \le K} \ \|\nabla f(x^k)\| \le \frac{(\nu+1)^{\nu} L_\nu \D_0^{\nu}}
     {2^{\frac{\nu}{2}} \gamma^{\frac{\nu}{2}} \nu^{\nu} K^{\frac{\nu}{2}}}.
\end{align*}
This completes the proof.

\subsection{Proof of Theorem \ref{thm:gcb convergence}}

    Let $\hat{\sigma}_f(t) \coloneqq \int_0^t \frac{\hat{\mu}(\tau)}{\tau} \d \tau$. By Theorem 1 in \cite{nesterov2025universal}, it holds that
    \begin{equation}\label{eq:gcb upper bound}
        \big|f(y) - f(x) - \langle \nabla f(x), y-x \rangle\big| \le 2\hat{\sigma}_f(t) - \hat{\mu}_f(t) + \frac{\hat{\mu}_f(t)}{t^2}\|y-x\|^2
    \end{equation}
    for any $x, y \in \reals^n$ and $t \in [0, \infty)$. We first establish the following upper bound for the gradient norm:
    \begin{equation}\label{eq:upper bound of gradient gcb}
        \|\nabla f(x)\| \le \inf_{t \ge 0} \ \frac{f(x) - f^\star + 2\hat{\sigma}_f(t)}{t},
    \end{equation}
    which can be viewed as a counterpart to the second inequality in Equation \eqref{eq:Holder smooth inequalities}. Fix $\zeta \in [0, \infty)$ and $x \in \reals^n$. By substituting $y = x - \zeta \cdot \frac{\nabla f(x)}{\|\nabla f(x)\|}$ and $t = \|x - y\| = \zeta$ into inequality \eqref{eq:gcb upper bound}, it yields
    \begin{align*}
        f(y) - f(x) + \zeta \|\nabla f(x)\| \le 2\hat{\sigma}_f(\zeta) - \hat{\mu}_f(\zeta) + \hat{\mu}_f(\zeta) = 2\hat{\sigma}_f(\zeta),
    \end{align*}
    which further implies $f(x) - f^\star \ge \zeta\|\nabla f(x)\| - 2\hat{\sigma}_f(\zeta)$ due to the fact that $f(y) \ge f^\star$. Since $\zeta$ is chosen arbitrarily, we have
    \begin{align*}
        f(x) - f^\star \ge \sup_{\zeta \ge 0} \ \zeta \|\nabla f(x)\| - 2\hat{\sigma}_f(\zeta) = (2\hat{\sigma}_f)^*(\|\nabla f(x)\|),
    \end{align*}
    where $(2\hat{\sigma}_f)^*$ is the Fenchel conjugate function of $2\hat{\sigma}_f$ (see Definition 3.22 in \cite{drusvyatskiy2020convex}). Consequently, by the Fenchel-Young inequality, it follows $\zeta \|\nabla f(x)\| \le 2\hat{\sigma}_f(\zeta) + (2\hat{\sigma}_f)^*(\|\nabla f(x)\|)$. Note that $f(x) - f^\star \ge (2\hat{\sigma}_f)^*(\|\nabla f(x)\|)$. We derive
    \begin{align*}
        \zeta \|\nabla f(x)\| \le 2\hat{\sigma}_f(\zeta) + f(x) - f^\star.
    \end{align*}
    Again, recall that $\zeta$ is arbitrary. This proves Equation \eqref{eq:upper bound of gradient gcb}. Since the function $f$ is convex, by Lemma 4 in \cite{nesterov2025universal}, it holds that $\hat{\sigma}_f(t) \le \hat{\mu}_f(t)$ for any $t \in [0, \infty)$. From Equation \eqref{eq:upper bound of gradient gcb}, we obtain
    \begin{align*}
         \|\nabla f(x)\| \le \inf_{t \ge 0} \ \frac{f(x) - f^\star + 2\hat{\mu}_f(t)}{t} \le \frac{f(x) - f^\star + 2\hat{\mu}_f(s_f(f(x) - f^\star))}{s_f(f(x) - f^\star)}.
    \end{align*}
    Since the complexity gauge is the inverse of the global curvature bound, it follows  $\hat{\mu}_f(s_f(f(x) - f^\star)) = f(x) - f^\star$, which leads to
    \begin{align*}
        \|\nabla f(x)\| \le \frac{3(f(x) - f^\star) }{s_f(f(x) - f^\star)}.
    \end{align*}
Combining the above inequality with Equation \eqref{eq:basic inequality}, we have
    \begin{align*}
        \D_{k+1}^2 \le \D_k^2 - \gamma(2-\gamma) \frac{\Delta_k^2}{\|\nabla f(x^k)\|^2} \le \D_k^2 - \frac{s_f^2(\Delta_k)}{9},
    \end{align*}
    where $\D_k = \dist(x^k, \cX^\star)$ and $\Delta_k = f(x^k) - f^\star$. Summing the relation over $k$ from $k = 0$ to $k = K-1$ gives
    \begin{align*}
        \frac{1}{K}\sum_{k=0}^{K-1} s_f^2(\Delta_k) \le \frac{9\D_0^2}{K}.
    \end{align*}
    It follows that 
    \begin{align*}
        \min_{1 \le k \le K} \ s_f(\Delta_k) \le \frac{3\D_0}{\sqrt{K}},
    \end{align*}
    which, by taking inverse on both sides, implies
    \begin{align*}
        \min_{1 \le k \le K} \ f(x^k) - f^\star \le \hat{\mu}_f\Big(\frac{3 \D_0}{\sqrt{K}}\Big).
    \end{align*}
Setting $K \ge (3\D_0/s_f(\varepsilon))^2$ ensures  $\hat{\mu}_f(3\D_0/\sqrt{K}) \le \varepsilon$ and completes the proof.

\subsection{Proof of Theorem \ref{thm:stochastic-polyak}}

Define $f^\star_\xi \assign \min_x f(x, \xi)$. The interpolation condition implies $\mathbb{E}[f^\star_\xi] = f^\star$. Using the update formula of the Polyak stepsize, we deduce that
\begin{align}
  \| x^{k + 1} - x^{\star} \|^2 ={} & \| x^k - x^{\star} \|^2 - 2 \gamma \cdot
  \alpha_k \langle \nabla f (x^k, \xi^k), x^k - x^{\star} \rangle + \gamma^2
  \cdot \alpha_k^2 \| \nabla f (x^k, \xi^k) \|^2 \nonumber\\
  \leq{} & \| x^k - x^{\star} \|^2 - 2 \gamma \cdot \alpha_k (f (x^k, \xi^k) -
  f^\star_{\xi^k}) + \gamma^2 \cdot \alpha_k^2 \| \nabla f (x^k, \xi^k) \|^2
  \nonumber\\
  ={} & \| x^k - x^{\star} \|^2 - \gamma (2 - \gamma) \frac{(f (x^k, \xi^k) -
  f^\star_{\xi^k})^2}{\| \nabla f (x^k, \xi^k) \|^2} . \nonumber
\end{align}
With the stochastic version of Equation \eqref{eq:Holder smooth inequalities}:
\[ \| \nabla f (x^k, \xi^k) \|^{\frac{\nu + 1}{\nu}} \leq \frac{\nu + 1}{\nu}
   L_{\nu}^{\frac{1}{\nu}} (f (x^k, \xi^k) - f^\star_{\xi^k}) \]
we deduce that
\begin{align}
  \| x^{k + 1} - x^{\star} \|^2 \leq{} & \| x^k - x^{\star} \|^2 - \gamma (2 -
  \gamma) \frac{(f (x^k, \xi^k) - f^\star_{\xi^k})^2}{\| \nabla f (x^k, \xi^k) \|^2}
  \nonumber\\
  \leq{} & \| x^k - x^{\star} \|^2 - \gamma (2 - \gamma) \frac{(f (x^k, \xi^k) -
  f^\star_{\xi^k})^2}{\left( \frac{\nu + 1}{\nu} \right)^{\frac{2 \nu}{\nu + 1}}
  L_{\nu}^{\frac{2}{\nu + 1}} (f (x^k, \xi^k) - f^\star_{\xi^k})^{\frac{2 \nu}{\nu
  + 1}}} \nonumber\\
  ={} & \| x^k - x^{\star} \|^2 - \frac{\gamma (2 - \gamma) \nu^{\frac{2
  \nu}{\nu + 1}}}{(\nu + 1)^{\frac{2 \nu}{\nu + 1}} L_{\nu}^{\frac{2}{\nu +
  1}}} (f (x^k, \xi^k) - f^\star_{\xi^k})^{\frac{2}{\nu + 1}} \nonumber
\end{align}
Taking the conditional expectation on both sides, we have
\begin{align}
  \mathbb{E} [\| x^{k + 1} - x^{\star} \|^2 |x^k] \leq{} & \| x^k - x^{\star}
  \|^2 - \frac{\gamma (2 - \gamma) \nu^{\frac{2 \nu}{\nu + 1}}}{(\nu +
  1)^{\frac{2 \nu}{\nu + 1}} L_{\nu}^{\frac{2}{\nu + 1}}} \mathbb{E} \left[
  (f (x^k, \xi^k) - f^\star_{\xi^k})^{\frac{2}{\nu + 1}} | x^k \right] \nonumber\\
  \leq{} & \| x^k - x^{\star} \|^2 - \frac{\gamma (2 - \gamma) \nu^{\frac{2
  \nu}{\nu + 1}}}{(\nu + 1)^{\frac{2 \nu}{\nu + 1}} L_{\nu}^{\frac{2}{\nu +
  1}}} \mathbb{E} [f (x^k, \xi^k) - f^\star_{\xi^k} | x^k ]^{\frac{2}{\nu + 1}}  \label{eqn:interpo-1} \\
  ={} & \| x^k - x^{\star} \|^2 - \frac{\gamma (2 - \gamma) \nu^{\frac{2
  \nu}{\nu + 1}}}{(\nu + 1)^{\frac{2 \nu}{\nu + 1}} L_{\nu}^{\frac{2}{\nu +
  1}}}[f (x^k) - f^{\star}]^{\frac{2}{\nu + 1}} \label{eqn:interpo-2} \\
  ={} & \| x^k - x^{\star} \|^2 - \frac{\gamma (2 - \gamma) \nu^{\frac{2
  \nu}{\nu + 1}}}{(\nu + 1)^{\frac{2 \nu}{\nu + 1}} L_{\nu}^{\frac{2}{\nu +
  1}}} \Delta_k^{\frac{2}{\nu + 1}}, \nonumber
\end{align}
where \eqref{eqn:interpo-1} uses Jensen's inequality $\mathbb{E} [ X^{\frac{2}{\nu + 1}}
] \geq \mathbb{E} [X]^{\frac{2}{\nu + 1}}$ for nonnegative random variable $X$ and $\nu \in (0, 1]$; \eqref{eqn:interpo-2} uses the interpolation condition $\mathbb{E}[f(x, \xi) - f_{\xi}^\star | x] = f(x) - f^\star$.With the H{\"o}lder growth condition, we have
\[ \mathbb{E} [\| x^{k + 1} - x^{\star} \|^2 |x^k] \leq \| x^k - x^{\star}
   \|^2 - \frac{\gamma (2 - \gamma) \nu^{\frac{2 \nu}{\nu + 1}}
   \rho_r^{\frac{2}{\nu + 1}}}{(\nu + 1)^{\frac{2 \nu}{\nu + 1}}
   L_{\nu}^{\frac{2}{\nu + 1}}} \| x^k - x^{\star} \|^{\frac{2 r}{\nu + 1}}.
\]
Taking total expectation and using $r \geq \nu + 1$, we have
\[ \mathbb{E} \left[ \| x^k - x^{\star} \|^{\frac{2 r}{\nu + 1}} \right]
   =\mathbb{E} \left[ (\| x^k - x^{\star} \|^2)^{\frac{r}{\nu + 1}} \right]
   \geq \mathbb{E} [\| x^k - x^{\star} \|^2]^{\frac{r}{\nu + 1}}, \]
which implies
\begin{align}
  \mathbb{E} [\| x^{k + 1} - x^{\star} \|^2] \leq{} & \mathbb{E} [\| x^k -
  x^{\star} \|^2] - \frac{\gamma (2 - \gamma) \nu^{\frac{2 \nu}{\nu + 1}}
  \rho_r^{\frac{2}{\nu + 1}}}{(\nu + 1)^{\frac{2 \nu}{\nu + 1}}
  L_{\nu}^{\frac{2}{\nu + 1}}} \mathbb{E} \left[ \| x^k - x^{\star}
  \|^{\frac{2 r}{\nu + 1}} \right] \nonumber\\
  \leq{} & \mathbb{E} [\| x^k - x^{\star} \|^2] - \frac{\gamma (2 - \gamma)
  \nu^{\frac{2 \nu}{\nu + 1}} \rho_r^{\frac{2}{\nu + 1}}}{(\nu + 1)^{\frac{2
  \nu}{\nu + 1}} L_{\nu}^{\frac{2}{\nu + 1}}} \mathbb{E} [\| x^k - x^{\star}
  \|^2]^{\frac{r}{\nu + 1}} . \nonumber
\end{align}
Define $\mathbb{D}_k \assign \mathbb{E} [\| x^k - x^{\star} \|^2]$. We have exactly the same recursive relation as in \eqref{eq:Holder smooth growth inequality}:
\[ \mathbb{D}_{k + 1} \leq \mathbb{D}_k - \frac{\gamma (2 - \gamma)
   \nu^{\frac{2 \nu}{\nu + 1}} \rho_r^{\frac{2}{\nu + 1}}}{(\nu + 1)^{\frac{2
   \nu}{\nu + 1}} L_{\nu}^{\frac{2}{\nu + 1}}} \mathbb{D}_k^{\frac{r}{\nu +
   1}}, \]
and the same results from the analysis of the deterministic case. Using
\[ \frac{\gamma (2 - \gamma) \nu^{\frac{2 \nu}{\nu + 1}}}{(\nu + 1)^{\frac{2
   \nu}{\nu + 1}} L_{\nu}^{\frac{2}{\nu + 1}}} \mathbb{E}
   [\Delta_k]^{\frac{2}{\nu + 1}} \leq \frac{\gamma (2 - \gamma) \nu^{\frac{2
   \nu}{\nu + 1}}}{(\nu + 1)^{\frac{2 \nu}{\nu + 1}} L_{\nu}^{\frac{2}{\nu +
   1}}} \mathbb{E} \big[ \Delta_k^{\frac{2}{\nu + 1}} \big] \leq
   \mathbb{D}_k -\mathbb{D}_{k + 1}, \]
following Theorem \ref{thm:Holder smooth growth} completes the proof.